\documentclass[onefignum,onetabnum]{siamonline250211}

\usepackage{lipsum}
\usepackage{amsfonts}
\usepackage{graphicx}
\usepackage{epstopdf}
\ifpdf
  \DeclareGraphicsExtensions{.eps,.pdf,.png,.jpg}
\else
  \DeclareGraphicsExtensions{.eps}
\fi

\usepackage{amsmath}

\usepackage{amssymb}
\usepackage{pifont}

\usepackage{booktabs}
\usepackage{subcaption}
\usepackage{tabularx}
\usepackage{multirow}
\usepackage{mathtools}
\usepackage{tikz}
\usetikzlibrary{positioning} 
\usetikzlibrary{calc}
\usepackage{enumitem}
\setlist[enumerate]{leftmargin=.5in}
\setlist[itemize]{leftmargin=.5in}

\newtheorem{assumption}[theorem]{Assumption}

\DeclareMathOperator*{\argmin}{arg\,min}

\newsiamremark{remark}{Remark}
\newsiamremark{hypothesis}{Hypothesis}
\crefname{hypothesis}{Hypothesis}{Hypotheses}
\newsiamthm{claim}{Claim}
\newsiamremark{fact}{Fact}
\crefname{fact}{Fact}{Facts}

\headers{Fixed-Point Neural Optimal Transport}{Y. Park, E. Gelphman, S. Osher and S. Wu Fung}

\title{Implicit Neural Optimal Transport via Fixed-Point Optimization\thanks{Yesom Park and Eric Gelphman contributed to this work equally. 
\funding{This work was supported by DARPA HR00112590074, DoE DE-SC0026262, NSF 2208272, ARO W911NF241015, and NSF DMS 2309810.}}}

\author{Yesom Park\thanks{Department of Mathematics, University of California, Los Angeles (\email{yeisom@math.ucla.edu}), \email{sjo@math.ucla.edu}.} 
    \and Eric Gelphman \thanks{Department of Applied Mathematics and Statistics, Colorado School of Mines (\email{eric\_gelphman@mines.edu}), \email{swufung@mines.edu}.} 
\and Stanley Osher\footnotemark[2] 
    \and Samy Wu Fung\footnotemark[3]}

\usepackage{amsopn}

\usepackage{algorithm}
\usepackage{algpseudocode}
\ifpdf
\hypersetup{
  pdftitle={Fixed-Point Neural Optimal Transport without Implicit Differentiation},
  pdfauthor={D. Doe, P. T. Frank, and J. E. Smith}
}
\fi

\usepackage{algpseudocode}

\begin{document}

\maketitle

\begin{abstract}
We propose an implicit neural formulation of optimal transport that eliminates adversarial min--max optimization and multi-network architectures commonly used in existing approaches. Our key idea is to parameterize a single potential in the Kantorovich dual and reformulate the associated c-transform as a proximal fixed-point problem. This yields a stable single-network framework in which dual feasibility is enforced exactly through proximal optimality conditions rather than adversarial training.
Despite the inner fixed-point computation, gradients can be computed without differentiating through the fixed-point iterations, enabling efficient training without requiring implicit differentiation. We further establish convergence of stochastic gradient descent. The resulting framework is efficient, scalable, and broadly applicable: it simultaneously recovers forward and backward transport maps and naturally extends to class-conditional settings.
Experiments on high-dimensional Gaussian benchmarks, physical datasets, and image translation tasks demonstrate strong transport accuracy together with improved training stability and favorable computational and memory efficiency.
\end{abstract}



\section{Introduction}

Optimal transport (OT) is a fundamental problem of finding a mapping between probability distributions that minimizes a prescribed transportation cost. Owing to its solid theoretical foundation and wide applicability, OT has been successfully employed in diverse fields such as traffic control \cite{carlier2008optimal,danila2006optimal, barthelemy2006optimal}, biomedical data analysis \cite{schiebinger2019optimal, koshizuka2022neural, bunne2023learning}, generative modeling \cite{wang2021deep,onken2021ot,zhang2023mean, liu2019wasserstein}, and domain adaptation \cite{courty2016optimal, courty2017joint, damodaran2018deepjdot, balaji2020robust}.

Recent advances in deep learning have led to a surge of interest in scalable OT solvers based on neural parameterizations. Early approaches are rooted in the Monge formulation \cite{lu2020large, xie2019scalable} and its relaxation to the Kantorovich framework \cite{makkuva2020optimal}. Despite their theoretical rigor, these formulations often entail significant computational complexity, particularly in high-dimensional settings. A predominant line of work is based on the Kantorovich dual formulation, which recasts the OT problem as a saddle-point optimization over transport maps and dual potentials \cite{liu2019wasserstein, taghvaei20192, korotin2021neural, liu2021learning, choi2024improving}. While this perspective enables scalable algorithmic implementations, it typically necessitates adversarial training of multiple neural networks, thereby introducing optimization instability, sensitivity to hyperparameter selection, and convergence difficulties. These challenges are further exacerbated in WGAN-based methods \cite{arjovsky2017wasserstein, liu2019wasserstein}, particularly as the dimensionality of the problem increases.

To leverage additional structural properties, several studies focus on the quadratic-cost setting, where Brenier’s theorem \cite{benamou2000computational} guarantees that the OT map can be expressed as the gradient of a convex potential. This observation has motivated the use of input convex neural networks (ICNNs) \cite{amos2017input} and related architectures \cite{taghvaei20192, korotin2021neural}. In this direction, prior work typically parameterizes either the primal and dual potentials separately~\cite{makkuva2020optimal, taghvaei20192}, or directly models the convex conjugate minimizer ~\cite{amos2022amortizing}. In addition, weak formulations have been explored to directly parameterize transport maps \cite{backhoff2019existence, korotin2022neural, asadulaev2022neural}. Nevertheless, many of these approaches continue to rely on auxiliary networks, alternating optimization procedures, or adversarial training, and thus inherit the limitations associated with min–max optimization.

An alternative line of research formulates OT as a dynamical system via continuous flows \cite{yang2020potential, tong2020trajectorynet, onken2021ot, huguet2022manifold}. These methods typically require solving ordinary differential equations (ODEs) or stochastic differential equations (SDEs), which imposes considerable computational overhead during both training and inference. Regularized variants, including entropic and $f$-divergence-based formulations \cite{genevay2016stochastic, seguy2017large, daniels2021score, gushchin2023entropic}, can improve numerical stability but generally introduce bias, leading to transport maps that deviate from the true OT solution.

Various efforts have been made to mitigate these challenges through improved optimization strategies, such as natural gradient methods \cite{shen2020sinkhorn, liu2024natural}, as well as regularization techniques including $L^2$ penalties \cite{genevay2016stochastic, sanjabi2018convergence} and cycle-consistency constraints \cite{genevay2016stochastic, sanjabi2018convergence, korotin2019wasserstein, korotin2021continuous}. However, these approaches do not fundamentally resolve the reliance on complex optimization schemes or multiple model components \cite{korotin2021neural, fan2021variational}.
More recently, formulations based on Hamilton–Jacobi–Bellman (HJB) equations \cite{park2025implicit} have been proposed to eliminate adversarial training by casting OT as a single-objective optimization problem \cite{park2025neural}. In particular, characteristic-based representations enable the recovery of both forward and backward transport maps within a unified framework. However, such methods require solving the HJB equation over the entire computational domain, which significantly limits their scalability in high-dimensional settings.

{\color{black}
In this work, we build upon single-potential formulation on quadratic OT \cite{taghvaei20192} and investigate an implicit neural framework that recovers both forward and backward transport maps through a single minimization problem defined over one neural network. 
Starting from the Kantorovich dual problem, we parameterize the dual potential as the negative of a convex function represented by a neural network $g_\theta$. Under this parameterization, the associated $c$-transform reduces to a strongly convex proximal optimization problem, whose minimizer can be computed efficiently via standard fixed-point iterations. This a single-objective optimization framework in which dual feasibility is enforced directly through the optimality conditions of the proximal operator, thereby eliminating the need for saddle-point formulations, auxiliary conjugate networks, and alternating optimization procedures. 

A central contribution of this work is the analysis of stochastic optimization under inexact inner fixed-point iteration. A fundamental challenge in implicit single-minimization OT formulations is that the convergence behavior of the model depends critically on the accuracy of the inner solver used to evaluate the $c$-transform. In practice, however, the associated inner problem cannot be solved exactly and must instead be approximated using a finite number of iterations.
We analyze the approximate stochastic gradient induced by this inexact fixed-point computation and establish convergence guarantees for the resulting SGD dynamics. In particular, we show that if the fixed-point residual is controlled within a prescribed tolerance, then the resulting approximate gradient remains a descent direction for the objective. Under standard diminishing step-size conditions, we further prove that convergence of SGD to a stationary point despite inexact inner solves. These results provide a theoretical justification for practical training with fixed-point iterations and remove the need for exact inner optimization in large-scale implicit neural OT.

We further extend the framework to class-conditional OT, where the objective is to learn transport maps that preserve class labels while minimizing transportation cost. We validate the proposed approach across a broad range of experimental settings. In particular, we demonstrate accurate recovery of transport maps on high-dimensional Gaussian benchmarks and strong performance on physics-based datasets involving complex structured distributions. We further evaluate the method on class-conditional transport tasks, including image translation experiments, where it achieves competitive performance while maintaining stable training dynamics and favorable computational and memory efficiency as the problem dimension increases.

The remainder of this paper is organized as follows. In Section \ref{sec:background}, we review the background on OT and related formulations. Section \ref{sec:method} introduces the implicit neural OT framework together with the convergence analysis of the resulting optimization procedure. In Section~\ref{sec:experiments}, we present experimental results across a range of benchmark datasets, including CCOT tasks and model ablation studies. Finally, Section \ref{sec:conclusion} concludes the paper and discusses future directions.}

\section{Background: Quadratic OT and Convex Potentials}\label{sec:background}

\subsection{Wasserstein--2 Distance and Kantorovich Duality}

Let $\mu$ and $\nu$ be probability measures on $\mathbb{R}^d$ with finite second moments. For the quadratic cost
\begin{equation}
c(x,z) = \frac{1}{2}\|x - z\|^2,
\label{eq:quadratic_cost}
\end{equation}
the squared $2$-Wasserstein distance is defined as
\begin{equation}
W_2^2(\mu,\nu)
=
\inf_{\gamma \in \Pi(\mu,\nu)}
\int_{\mathbb{R}^d \times \mathbb{R}^d}
\frac{1}{2}\|x - z\|^2 \, d\gamma(x,z),
\end{equation}
where $\Pi(\mu,\nu)$ denotes the set of couplings with marginals $\mu$ and $\nu$.

The Kantorovich dual formulation states that
\begin{equation}
W_2^2(\mu,\nu)
=
\sup_{\varphi}
\left\{
\mathbb{E}_{x\sim\mu}[\varphi(x)]
+
\mathbb{E}_{z\sim\nu}[\varphi^c(z)]
\right\},
\end{equation}
where the $c$-transform is defined by
\begin{equation}
\varphi^c(z)
=
\inf_{y \in \mathbb{R}^d}
\left\{
\frac{1}{2}\|y - z\|^2 - \varphi(y)
\right\}.
\end{equation}

\subsection{Brenier's Theorem and Convex Potentials}

For the quadratic cost, OT admits additional structure. 
Brenier's theorem states that if $\mu$ is absolutely continuous, then there exists a convex function $u : \mathbb{R}^d \to \mathbb{R}$ such that the OT map from $\mu$ to $\nu$ is given by
\begin{equation}
T_{\mu\to\nu}(x) = \nabla u(x).
\end{equation}
Equivalently, writing $u(x)=\frac12\|x\|^2 + g(x)$ with $g$ convex, the map can be expressed as
\begin{equation}
T_{\mu\to\nu}(x) = x + \nabla g(x), 
\label{eq:T}
\end{equation}
which is a result of the following theorem. 

\begin{theorem}
Suppose $\mu$ is absolutely continuous with respect to the Lebesgue measure on $\mathbb{R}^d$. Then, $\exists$ a unique OT map $T: \mathbb{R}^d \rightarrow \mathbb{R}^d$ with respect to the quadratic cost \eqref{eq:quadratic_cost} where $T_{\#}\mu = \nu$ and $T$ is given by \eqref{eq:T}.
\label{theorem:optimal_transport_map}
\end{theorem}
This is a standard result from~\cite[Theorem 10.28]{villani2008optimal} and we provide a proof in the appendix for completeness.

This characterization motivates parameterizing the dual potential as
\begin{equation}
\varphi(x) = -g(x),
\end{equation}
where $g$ is convex. Under this choice, the $c$-transform becomes
\begin{equation}\label{eq:c_transform}
(-g)^c(z)
=
\inf_{y \in \mathbb{R}^d}
\left\{
\frac{1}{2}\|y - z\|^2 + g(y)
\right\}.
\end{equation}
This expression is precisely the Moreau envelope (or quadratic inf-convolution) of $g$~\cite{tibshirani2025laplace, osher2023hamilton} and corresponds to the solution of a Hamilton-Jacobi equation~\cite{darbon2021connecting, darbon2021bayesian, heaton2024global, di2026operator}.

\subsection{Proximal Characterization and Transport Maps}

For convex and differentiable $g$, the minimizer
\begin{equation}
y^*(z)
=
\arg\min_y
\left\{
\frac{1}{2}\|y - z\|^2 + g(y)
\right\}
\end{equation}
is uniquely defined~\cite{parikh2014proximal}. 
The forward and backward transport maps admit the representations
\begin{equation}
T_{\mu\to\nu}(x)
=
x + \nabla g(x),
\qquad
T_{\nu\to\mu}(z)
=
y^*(z).
\end{equation}
Therefore, the quadratic OT problem can be formulated entirely in terms of a single convex potential $g$, whose $c$-transform is evaluated by solving a strongly convex minimization problem. This proximal characterization forms the basis of the implicit neural formulation developed in the next section.

\section{Implicit Neural OT}\label{sec:method}

We now present an implicit neural formulation of quadratic OT based on the proximal characterization of the $c$-transform.

\subsection{Parameterized Dual Formulation}

Recall the Kantorovich dual formulation for the quadratic cost:
\begin{equation}
W_2^2(\mu,\nu)
=
\sup_{\varphi}
\left\{
\mathbb{E}_{x\sim\mu}[\varphi(x)]
+
\mathbb{E}_{z\sim\nu}[\varphi^c(z)]
\right\}.
\end{equation}
We parameterize the dual potential as
\begin{equation}
\varphi_\theta(x) = -g_\theta(x),
\end{equation}
where $g_\theta : \mathbb{R}^d \to \mathbb{R}$ is assumed to be convex in its argument. Under this parameterization, the $c$-transform becomes
\begin{equation}
(-g_\theta)^c(z)
=
\inf_{y \in \mathbb{R}^d}
\left\{
\frac{1}{2}\|y - z\|^2 + g_\theta(y)
\right\}.
\end{equation}
Substituting into the dual objective yields
\begin{equation}
W_2^2(\mu,\nu)
=
\sup_\theta
\left\{
-
\mathbb{E}_{x\sim\mu}[g_\theta(x)]
+
\mathbb{E}_{z\sim\nu}
\left[
\inf_y
\left(
\frac{1}{2}\|y - z\|^2 + g_\theta(y)
\right)
\right]
\right\}.
\end{equation}
Equivalently, we minimize the negative dual objective:
\begin{equation}
\label{eq:dual_loss}
\mathcal{L}(\theta)
=
\mathbb{E}_{x\sim\mu}[g_\theta(x)]
-
\mathbb{E}_{z\sim\nu}
\left[
\inf_{y}
\left(
\frac{1}{2}\|y - z\|^2 + g_\theta(y)
\right)
\right].
\end{equation}

\subsection{Constrained Implicit Formulation}

For each $z \in \mathbb{R}^d$, define the proximal minimizer
\begin{equation}
\label{eq:prox_def}
y_\theta^\star(z)
=
\arg\min_{y \in \mathbb{R}^d}
\left\{
\frac{1}{2}\|y - z\|^2 + g_\theta(y)
\right\}.
\end{equation}

Since the objective is strongly convex in $y$, the minimizer exists and is unique. The $c$-transform can therefore be written explicitly as
\begin{equation}
(-g_\theta)^c(z)
=
\frac{1}{2}\|y_\theta^\star(z) - z\|^2
+
g_\theta\big(y_\theta^\star(z)\big).
\end{equation}
Substituting this expression into \eqref{eq:dual_loss}, the training problem becomes the constrained training problem given by
\begin{equation}
\label{eq:constrained_problem}
\begin{aligned}
\min_\theta \quad 
& 
\mathbb{E}_{x\sim\mu}[g_\theta(x)]
-
\mathbb{E}_{z\sim\nu}
\left[
\frac{1}{2}\|y_\theta^\star(z) - z\|^2
+
g_\theta\big(y_\theta^\star(z)\big)
\right] \\
\text{s.t.} \quad
&
y_\theta^\star(z)
=
\arg\min_{y}
\left\{
\frac{1}{2}\|y - z\|^2 + g_\theta(y)
\right\}
\quad
\text{for } z \sim \nu.
\end{aligned}
\end{equation}
This formulation is closely related to single-potential dual representations of quadratic OT~\cite{taghvaei20192}. Our contribution is to leverage this structure to develop a scalable and practically effective method that performs well across a wide range of datasets, together with a convergence analysis of the resulting optimization procedure under inexact evaluation of $y_\theta^\star(z)$.

In this formulation, learning reduces to optimizing a single convex potential $g_\theta$, with the model output implicitly defined by the proximal optimality condition that characterizes $y_\theta^\star(z)$. Importantly, dual feasibility is enforced by construction through~\eqref{eq:prox_def}, eliminating the need for adversarial min--max formulations with auxiliary networks~\cite{korotin2019wasserstein, korotin2022neural, heaton2022wasserstein, lin2021alternating}, as well as time-stepping approaches that require learning entire trajectories~\cite{onken2021ot, ruthotto2020machine}.



Once $g_\theta$ is trained, the forward and backward transport maps are given by
\begin{equation}
T_{\mu\to\nu}(x)
=
x + \nabla g_\theta(x),
\qquad
T_{\nu\to\mu}(z)
=
y_\theta^\star(z).
\end{equation}

Both transport directions are therefore represented using a single convex potential. The backward map is computed as the unique minimizer of a strongly convex objective, while the forward map follows directly from Brenier's theorem. The overall training procedure, including the gradient computation described in the following subsection, is summarized in Algorithm~\ref{alg:implicit_ot}.

\begin{algorithm}[t]
\caption{Implicit Neural OT via Fixed-Point Proximal Updates} \label{alg:implicit_ot}
\begin{algorithmic}[1]
\Require Datasets $\mathcal{D}_\mu, \mathcal{D}_\nu$, initial parameters $\theta$, step size $\eta$, fixed-point steps $K$
\For{each training iteration}
    
    \State Sample minibatches $\{x_i\}_{i=1}^B \sim \mathcal{D}_\mu$, $\{z_i\}_{i=1}^B \sim \mathcal{D}_\nu$

    \State \textbf{// Fixed-point iteration for proximal operator}
    \For{each $z_i$ in minibatch}
        \State Initialize $y^{(0)} \gets z_i$
        \For{$k = 0, \dots, K-1$}
            \State $y^{(k+1)} \gets y^{(k)} - \alpha \big( \nabla_y g_\theta(y^{(k)}) + y^{(k)} - z_i \big)$
        \EndFor
        \State $\tilde{y}_i \gets \texttt{stop\_gradient}(y^{(K)})$ // Detach computational graph
    \EndFor

    \State \textbf{// Compute Loss}
    \State $g \gets \frac{1}{B} \sum_{i=1}^B \nabla_\theta g_\theta(x_i) - \frac{1}{B} \sum_{i=1}^B \nabla_\theta g_\theta(\tilde{y}_i)$

    \State \textbf{// Parameter update}
    \State $\theta \gets \theta - \eta \, g$

\EndFor
\end{algorithmic}
\end{algorithm}

\begin{remark}[PDE interpretation]
The proposed formulation admits an alternative interpretation from the perspective of partial differential equations (PDEs). In particular, the optimality condition of quadratic OT can be expressed in terms of the Hamilton--Jacobi (HJ) equation with a quadratic Hamiltonian. Owing to convexity, its solution can be characterized by the Hopf--Lax formula, which coincides with the $c$-transform in \eqref{eq:prox_def}. From this viewpoint, our method can be interpreted as solving the HJ equation via fixed-point iterations applied to the Hopf--Lax operator \cite{park2025fixed}. 

In contrast to prior works, which solve the HJ equation over the entire spatio-temporal computational domain and are therefore computationally demanding in high dimensions, we instead evaluate a local fixed-point map that enforces the corresponding optimality condition. This provides a causality-free mechanism for enforcing the PDE optimality condition without requiring the propagation of full PDE dynamics. Consequently, the proposed approach yields a significantly more efficient procedure while still enforcing the underlying optimality condition associated with the OT problem.
\end{remark}

\subsection{Gradient Computation Without Implicit Differentiation}
\label{subsec:gradient_computation}
{\color{black}
Since the model output is defined implicitly through a proximal optimization layer, one might expect that gradient computation requires implicit differentiation techniques as in implicit deep learning frameworks~\cite{el2021implicit, fung2022jfb, gelphman2025end, gelphman2026convergence}. However, the structure of formulation~\eqref{eq:constrained_problem} allows us to compute the gradient of the objective $\mathcal{L}$ without differentiating through the inner minimization that defines $y_\theta^\star(z)$. This cancellation-type property has also been observed in prior works~\cite{amos2022amortizing, taghvaei20192}, where the structure of the convex conjugate eliminates the need for implicit differentiation. Below, we show the formula for such derivative and provide its proof for completeness.}

\begin{lemma}
\label{lem:no_implicit_diff}
Assume $0 < \gamma < 1$, $g_\theta(y)$ is $\gamma$-weakly convex, continuously differentiable in $y$, and L-smooth in $\theta$, and that $y_\theta^\star(z)$ is the unique minimizer of
\begin{equation}
\min_y \left\{ \frac{1}{2}\|y - z\|^2 + g_\theta(y) \right\}.
\end{equation}
Then the gradient of $\mathcal{L}$ is given by
\begin{equation}
\label{eq:gradient_formula}
\frac{d\mathcal{L}}{d\theta}
=
\mathbb{E}_{x \sim \mu}
\left[
\frac{\partial g_\theta(x)}{\partial \theta}
\right]
-
\mathbb{E}_{z \sim \nu}
\left[
\frac{\partial g_\theta\big(y^\star_\theta(z)\big)}{\partial \theta}
\right].
\end{equation}
\end{lemma}

\begin{proof}
The forward pass of the network can be characterized as finding the fixed point of the operator
\begin{equation}
F_{\theta}(y) = y - \alpha(\nabla g + y - z) 
\end{equation}
Assume that we can interchange $\mathbb{E}_{x}[\cdot]$ and $\mathbb{E}_{z}[\cdot]$ for any $x,z$. Then,
\begin{align}
\frac{d L}{d \theta} &= \mathbb{E}_{x \sim \mu}\left[\frac{d g}{d \theta}\right] - \mathbb{E}_{z \sim \nu}\left[\frac{d}{d \theta}(\frac{1}{2}\|y^* - z\|^2 + g(y^*))\right] \\
&=  \mathbb{E}_{x \sim \mu}\left[\frac{d g}{d \theta}\right] - \mathbb{E}_{z \sim \nu}\left[(y^* - z)\frac{d y^*}{d \theta} + \frac{ \partial g }{\partial y}\frac{d y^*}{d \theta} + \frac{\partial g}{\partial \theta}\right] \\
&= \mathbb{E}_{x \sim \mu}\left[\frac{\partial g}{\partial \theta}\right] - \mathbb{E}_{z \sim \nu}\left[\left((y^* - z) + \frac{ \partial g }{\partial y}\right)\frac{d y^*}{d \theta} + \frac{\partial g}{\partial \theta}\right].
\end{align}
If $y^*$ is a minimizer of $\frac{1}{2}\|y-z\|^2 + g_{\theta}(y)$, then
\begin{equation*}
\left( \nabla_{y}\left(\frac{1}{2}\|y-z\|^2 + g_{\theta}(y)\right)\right|_{y=y^*} = \left(2(\frac{1}{2}(y-z)) + \frac{\partial g(y)}{\partial y}\right|_{y=y^*} = y^* - z + \frac{\partial g(y^*)}{\partial y} = 0.
\end{equation*}
Thus, if the fixed point problem is solved exactly,
\begin{equation}
\nabla_{\theta}L = \mathbb{E}_{x \sim \mu}\left[\frac{\partial g(x)}{\partial \theta}\right] - \mathbb{E}_{z \sim \nu}\left[\frac{\partial g(y^*)}{\partial \theta}\right].
\end{equation}
\end{proof}

Lemma~\ref{lem:no_implicit_diff} shows that the gradient of the dual objective depends only on explicit derivatives of $g_\theta$, evaluated at samples from $\mu$ and at the proximal points $y_\theta^\star(z)$. In particular, no implicit differentiation, i.e., computation of $\frac{dy_\theta^\star}{d\theta}$ through the inner optimization problem, is required. This cancellation follows directly from the first-order optimality condition of the proximal operator and can be interpreted as an application of the envelope theorem. Consequently, backpropagation can be performed without differentiating through the inner fixed-point iterations used to compute $y_\theta^\star(z)$, making gradient computation computationally efficient in practice.


\newcommand{\dgytilde}{\frac{\partial g(\tilde{y}_{\theta}(z))}{\partial \theta}}
\newcommand{\dgystar}{\frac{\partial g(y^*_{\theta}(z))}{\partial \theta}}
\newcommand{\dgx}{\frac{\partial g(x_{\theta})}{\partial \theta}}
\newcommand{\ytilde}{\tilde{y}_{\theta}(z)}
\newcommand{\ystar}{y_{\theta}^*(z)}
\newcommand{\xitheta}{\xi_{\theta}(z)}

\subsection{Convergence}
{\color{black}
In this section, we establish that stochastic gradient descent converges to a stationary point of the loss function even when the inner fixed-point problem is solved only approximately up to a prescribed error tolerance.

In the idealized setting where the fixed-point problem is solved exactly, Lemma~\ref{lem:no_implicit_diff} shows that implicit differentiation is not required for computing gradients with respect to $\theta$. However, this assumption does not hold in practice, as the proximal fixed-point problem must be solved using a finite number of iterations, leading to inexact inner solutions.
As a result, the optimization dynamics are governed by approximate proximal evaluations, and analyzing their effect on convergence is essential.

For ease of presentation, the detailed proofs are deferred to Appendix~\ref{appen:proof}.
}

\subsubsection{Approximate Gradient}
Suppose the fixed point problem is not solved exactly, i.e. instead of the true fixed point $y^*$ an approximate fixed point $\tilde{y}$ is computed.
Then, the stochastic gradient $\tilde{d}_{\theta}$ is given by
\begin{equation}
    \tilde{d}_{\theta} = \mathbb{E}_{x \sim \mu}\left[\frac{\partial g(x)}{\partial \theta}\right] - \mathbb{E}_{z \sim \nu}\left[\frac{\partial g(\tilde{y})}{\partial \theta}\right].
    \label{eq:d_tilde}
\end{equation}

\subsubsection{Preliminary Assumptions and Fixed Point Computation Lemma}

\begin{assumption}
The objective function $L(\theta)$ is bounded from below by $L_{inf}$ in some open subset of its domain.. The function $g$ is $C^1$ with respect to all variables and the gradients of $g$ with respect to $\theta$ and $y$, are $L_\theta$- and $L_y$-Lipschitz, respectively. Furthermore, assume the $2^{nd}$ order partial derivatives of $g$ with respect to $\theta$ exist.
\label{assumption:g_lipschitz}
\end{assumption}

From this point onwards, denote $x \sim \mu$ and $z \sim \nu$ as just $x$ and $z$, respectively.

\newcommand{\lemmaFixedPoint}[1]{
Under Assumption ~\ref{assumption:g_lipschitz}, suppose the approximate fixed point $\tilde{y}$ satisfies
\begin{equation}
\|(\tilde{y} - z) + \nabla_{y}g_{\theta}(\tilde{y})\| < \epsilon_p.
\end{equation}
Then, $\exists \epsilon > 0$ such that
\begin{equation}
\|y^* - \tilde{y}\| < \epsilon.
\end{equation}}
\begin{lemma}
\label{lemma:approx_fixed_point}
\lemmaFixedPoint{main}
\end{lemma}

\subsubsection{Upper Bound on $2^{nd}$ Moments of Exact and Approximate Gradients}
\newcommand{\lemmasecondmoment}[1]{
Under the assumptions of Lemma ~\ref{lemma:approx_fixed_point} along with $\int\|x\|^2d\mu(x) < \infty$ and $\int\|z\|^2d\nu(z) < \infty$,
\begin{equation}
\mathbb{E}_x\left[\left\|\frac{\partial g(x_{\theta})}{\partial \theta}\right\|^2\right] < \infty, \ \mathbb{E}_z\left[\left\|\frac{\partial g(\tilde{y_{\theta}}(z))}{\partial \theta}\right\|^2\right] < \infty, \text{ and } \mathbb{E}_z\left[\left\|\frac{\partial g(y_{\theta}^*(z))}{\partial \theta}\right\|^2\right] < \infty.
\end{equation}}

\begin{lemma}
\label{lemma:2nd_moment}
\lemmasecondmoment{main}
\end{lemma}

\newcommand{\theoremnormsquared}[1]{
Under the assumptions of Lemma ~\ref{lemma:2nd_moment} The norm of the true gradient $\nabla_{\theta}L$ and the stochastic gradient $\tilde{d}_{\theta}$ squared is bounded above, i.e. $\exists 0 < \tilde{M} < +\infty$ and $\exists 0 < M_{L} < +\infty$ such that $\forall \theta$
\begin{equation}
\|\tilde{d}_{\theta}\|^2 = \left\|\mathbb{E}_x\left[\dgx\right] - \mathbb{E}_z\left[\dgytilde\right]\right\|^2 \leq \tilde{M}
\end{equation} and
\begin{equation}
\|\nabla_{\theta}L\|^2 = \left\|\mathbb{E}_x\left[\dgx\right] - \mathbb{E}_z\left[\dgytilde\right]\right\|^2 \leq M_L.
\end{equation}}

\begin{theorem}
\label{theorem:norm_squared}
\theoremnormsquared{main}
\end{theorem}

\subsubsection{Approximate Gradient is a Descent Direction}
\newcommand{\theoremlowerboundinnerproduct}[1]{
Under the assumptions of Theorem ~\ref{theorem:norm_squared}, suppose the norm of the difference between the exact and computed fixed point satisfies, $\forall \theta$ and $\forall z, \  \|\tilde{y}_{\theta}(z) - y_{\theta}^*(z)\| < \epsilon \leq \frac{V\|\nabla_{\theta}L\|^2}{L_{\theta}\left(\sqrt{M_x} + \sqrt{M_z}\right)}$, for some $0 < V < 1$. Then, $\exists U = 1-V > 0$ such that $\forall \theta$ 
\begin{equation}
\langle\nabla_{\theta}L,\tilde{d}_{\theta} \rangle \geq U\|\nabla_{\theta}L\|^2.
\end{equation}}

\begin{theorem}
\theoremlowerboundinnerproduct{main}
\label{theorem:lower_bound_on_inner_product}
\end{theorem}

\subsubsection{Convergence Results}
We begin by introducing notation to formalize the stochasticity arising in the training process. Let $\{\xi_j\}_{j\ge0}$ denote a sequence of independent random variables representing the sampling procedure used to construct the stochastic gradient $\tilde{d}_{\xi_j}(\theta)$. In particular, $\xi_j$ corresponds to the random draw of initial conditions $x \sim \rho$ used to compute the stochastic gradient update at iteration $j$.

We analyze the convergence of SGD when $\tilde{d}_{\theta}$ is used as a stochastic gradient surrogate. Specifically, we consider the iterative scheme
\begin{equation}
\theta_{j+1} = \theta_j - \alpha_j \tilde{d}_{\xi_j}(\theta_j), \qquad j \ge 0,
\label{eq:SGD_iteration}
\end{equation}
for minimizing the loss function over $\theta \in \mathbb{R}^p$. Here, $\tilde{d}_{\xi_j}(\theta_j)$ denotes the JFB update computed using either a single sample or a minibatch of samples with corresponding learning rate $\alpha_j$ at iteration $j$.

Following the notation of ~\cite{bottou2018optimization}, we use $\mathbb{E}_{\xi_j}[\cdot]$ to denote the conditional expectation with respect to the randomness at iteration $j$, given the current iterate $\theta_j$. Since $\theta_j$ depends on the sequence of random variables $\{\xi_0, \xi_1, \ldots, \xi_{j-1}\}$, we also consider the total expectation of the objective with respect to all prior randomness, which we write as
\begin{equation}
\mathbb{E}[\mathbb{E}_x[J_x(\theta_j)]]
=
\mathbb{E}_{\xi_0}\Big[
\mathbb{E}_{\xi_1}\big[
\cdots
\mathbb{E}_{\xi_{j-1}}\big[
\mathbb{E}_x[J_x(\theta_j)]
\big]
\cdots
\big]
\Big].
\end{equation}

With this notation in place, we establish the following Lemma, which is used to prove the main result.

\newcommand{\lemmaExpectationDescent}[1]{
Under the assumptions of Theorem ~\ref{theorem:lower_bound_on_inner_product}, the SGD iterations ~\eqref{eq:SGD_iteration} satisfy
\begin{equation}
\mathbb{E}_{\xi_j}[L(\theta_{j+1})] - L(\theta_j) \leq -\alpha_j U \|\nabla_{\theta}L(\theta_j)]\|^2 + \alpha_j^2L_{\theta}\tilde{M}.
\end{equation}
If $0 < \alpha_j \leq \frac{UM_L}{L_{\theta} \tilde{M}}$, then it follows that $\mathbb{E}_{\xi_j}[L(\theta_{j+1})] - L(\theta_j) \leq 0$.}
\begin{lemma}
\label{lemma:expectation_descent}
\lemmaExpectationDescent{main}
\end{lemma}

With this result, the main results of this paper can be proven.

\newcommand{\theoremExpectationCesaroConvergence}[1]{
Suppose the sequence of learning rates $\{\alpha_j\}_{j=0}^{\infty}$ is monotonically decreasing and satisfies $\sum_{j=0}^{\infty} \alpha_j = \infty$, $\sum_{j=0}^{\infty} \alpha_{j}^2 < \infty$, and $0 < \alpha_0 \leq \frac{UM_L}{L_{\theta} \tilde{M}}$. Let $A_K= \sum_{j=0}^{K-1} \alpha_j$. Then, under the assumptions of Lemma ~\ref{lemma:expectation_descent} the SGD iteration \eqref{eq:SGD_iteration} satisfies
\begin{equation*}
\lim_{K \rightarrow \infty} \mathbb{E} \left[ \frac{1}{A_K}\sum_{j=0}^{K} \alpha_j \left\| \nabla_{\theta}L(\theta_j) \right\|^2 \right] = 0.
\end{equation*}}

\begin{theorem}
\label{theorem:expectation_cesaro_convergence}
\theoremExpectationCesaroConvergence{main}
\end{theorem}

In other words, the weighted Cesaro sum of the sequence $\{ \left\| \nabla_{\theta}L(\theta_j)\right\|^2\}_{j=0}^{\infty}$ converges in (total) expectation to 0. 
Using Theorem~\ref{theorem:expectation_cesaro_convergence}, one can then use standard SGD analysis to show the following theorem and corollary.

\newcommand{\theoremExpectationLiminf}[1]{
Under the assumptions of Theorem ~\ref{theorem:expectation_cesaro_convergence}, the SGD iteration \eqref{eq:SGD_iteration} satisfies
\begin{equation*}
\liminf_{j \rightarrow \infty} \mathbb{E}\left[\left\| \nabla_{\theta}L(\theta_j)\right\|^2 \right] = 0.
\end{equation*}
}

\begin{theorem}
\label{theorem:expectation_liminf}
\theoremExpectationLiminf{main}
\end{theorem}
Using Theorem ~\ref{theorem:expectation_cesaro_convergence}, we can also prove convergence in probability to a critical point.
\newcommand{\corollaryConvergenceProbability}[1]{
Suppose the assumptions of Theorem ~\ref{theorem:expectation_cesaro_convergence} hold. For any $K \in \mathbb{N}$ let $j(K) \in \{0,1,...,K\}$ represent a random index chosen with probabilities proportional to $\{\alpha_j\}_{j=0}^{K}$. Then, $\{\left\| \nabla_{\theta}L(\theta_j)\right\|\}_{j=0}^{K} \rightarrow 0$ as $K \rightarrow \infty$ in probability.}
\begin{corollary}
\label{corollary:convergence_probability}
\corollaryConvergenceProbability{main}
\end{corollary}

\subsection{Class Conditional OT}

In many applications, the source and target measures admit class-wise decompositions
\begin{equation}
\mu = \sum_{k=1}^K \pi_k \mu_k,
\qquad
\nu = \sum_{k=1}^K \pi_k \nu_k,
\end{equation}
where \(\mu_k\) and \(\nu_k\) denote the distributions restricted to class \(k\).  
class-conditional OT (CCOT) enforces that mass is transported only within corresponding classes by solving, for each \(k\),
\begin{equation}
T_k
=
\argmin_{T_\# \mu_k = \nu_k}
\mathbb{E}_{x\sim \mu_k} \Big[ \tfrac12 \| x - T(x) \|^2 \Big].
\end{equation}

Rather than defining separate potentials \(\{ g_k \}_{k=1}^K\) or assuming well-separated supports, we parameterize a single class-conditional potential $g_\theta(x, k) : \mathbb{R}^d \times \{0, \dots, K-1\} \to \mathbb{R}$,
where \(k\) is a class label. The network receives the input \(x\) concatenated with a one-hot encoding of \(k\).

Then, the class-conditional dual objective becomes
\begin{equation}
\label{eq:conditional_potential_loss}
\mathcal{L}_{\mathrm{CC}}(\theta)
=
\frac{1}{K} \sum_{k=0}^{K-1} 
\Bigg[
\mathbb{E}_{x \sim \mu_k} \big[ g_\theta(x, k) \big]
-
\mathbb{E}_{z \sim \nu_k} \Big[ \frac12 \| z - y_\theta^\star(z, k) \|^2 + g_\theta(y_\theta^\star(z, k), k) \Big]
\Bigg],
\end{equation}
where the backward OT map for class \(k\) is defined as the fixed-point solution
\begin{equation}
y_\theta^\star(z, k) = \argmin_{y} \Big\{ \frac{1}{2} \| y - z \|^2 + g_\theta(y, k) \Big\}, 
\quad z \sim \nu_k.
\end{equation}

Hence, conditioning the potential on class labels allows a single network to naturally represent class-conditional transport maps while maintaining a unified convex potential.

\section{Experiments}\label{sec:experiments}
We evaluate the proposed method across a variety of datasets and experimental settings to assess its performance and efficiency. We compare against existing baseline methods and further conduct ablation studies to analyze the contribution of each component. 

All experiments are implemented in PyTorch and detailed implementation details are provided in Appendix~\ref{appen:implementation_details}. 
The experiments in Section~\ref{sec:mnist_fmnist} are conducted on an NVIDIA RTX Blackwell 6000 GPU, while all remaining experiments are performed on a single NVIDIA TITAN V GPU (12GB).
The implementation of our method is publicly available at:
\url{https://github.com/Yebbi/ImplicitOT}

\subsection{Evaluation on High-dimensional Gaussian Distributions}\label{sec:high_dim_gaussians_rewrite}

Quantitative evaluation of OT methods for general distributions is often challenging due to the lack of closed-form solutions. To address this, we focus on Gaussian distributions, $\mu = \mathcal{N}(\mathbf{0}, \Sigma_\mu)$ and $\nu = \mathcal{N}(\mathbf{0}, \Sigma_\nu)$, for which the OT map admits an analytical solution:
\begin{equation}\label{eq:OT_Gaussian}
T^{\nu*}_{\mu}(x) = \Sigma_\mu^{-\frac12} (\Sigma_\mu^{\frac12} \Sigma_\nu \Sigma_\mu^{\frac12})^{\frac12} \Sigma_\mu^{-\frac12}x.
\end{equation}
Following the protocol of \cite{korotin2021neural}, we consider dimensions $d \in [2, 64]$, constructing $\Sigma_\mu$ and $\Sigma_\nu$ with random orthonormal eigenvectors and eigenvalues whose logarithms are sampled uniformly from $[-2, 2]$.

\paragraph{Baseline Methods} 
We compare our method with representative OT approaches, including NOT \cite{korotin2022neural}, which directly parameterizes the transport map, MM-v1 \cite{taghvaei20192, korotin2021neural} and MM \cite{korotin2021neural}, which adopt min--max formulations with and without convexity constraints, respectively, NCF \cite{park2025neural}, which constructs transport maps via the HJ equation, LS \cite{seguy2017large}, a dual OT solver with entropic regularization, and WGAN-QC \cite{liu2019wasserstein}, which optimizes a Wasserstein objective under a quadratic cost. Except for NOT, all methods use the same network architecture as ours.


\paragraph{Evaluation Metrics} 
Performance is evaluated using the \emph{unexplained variance percentage (UVP)} \cite{korotin2019wasserstein}. 
Given a predicted transport map $\hat{T}:\mu \to \nu$ and the ground-truth OT map $T^\ast$, the UVP is defined as
\begin{equation}
\mathcal{L}^2\text{-UVP}\left(\hat{T}\right)
\coloneqq
100 \, \frac{\left\lVert \hat{T} - T^\ast \right\rVert_{L^2(\mu)}}{\mathrm{Var}(\nu)} \; (\%).
\end{equation}
We further report computational efficiency in terms of training and inference time, peak memory consumption, and storage requirements for bidirectional transport maps. 

\paragraph{Results} 
\begin{table}[t]
\centering
\vspace{-2mm}
\caption{\textit{Quantitative evaluation on Gaussian distributions.} UVP $(\downarrow)$  is measured across different OT methods as the data dimension $d$ increases. }
\label{tab:gaussians_uvp}
\scalebox{0.86}{
\begin{tabular}{lcccccc}
\toprule
Method & $d=2$ & $d=4$ & $d=8$ & $d=16$ & $d=32$ & $d=64$ \\
\midrule
NOT &   77.248 & 125.419 & 114.056 & 176.086 & 182.287 & 196.831\\
WGAN-QC   &   1.596 & 5.897 & 31.0367 &  59.314 & 113.237 & 141.407\\
LS    & 5.806 & 9.781 & 15.963 & 25.232 & 41.445 & 55.360 \\
MM-v1   & 0.161 & 0.172 &  0.173 & 0.210 & 0.374 & 0.415 \\
MM:R & 0.012 & 0.048 & 0.117 & 0.202 & 0.354 & 0.604 \\
NCF & \textbf{0.010} & 0.021 & 0.086 & 0.146 & 0.436 & 0.858\\
\textbf{Ours} & 0.013 & \textbf{0.016} & \textbf{0.046} & \textbf{0.053} &  \textbf{0.054}  & \textbf{0.0822} \\
\bottomrule
\end{tabular}}
\vspace{-1em}
\end{table}

\begin{figure}
    \centering
    \includegraphics[width=\linewidth]{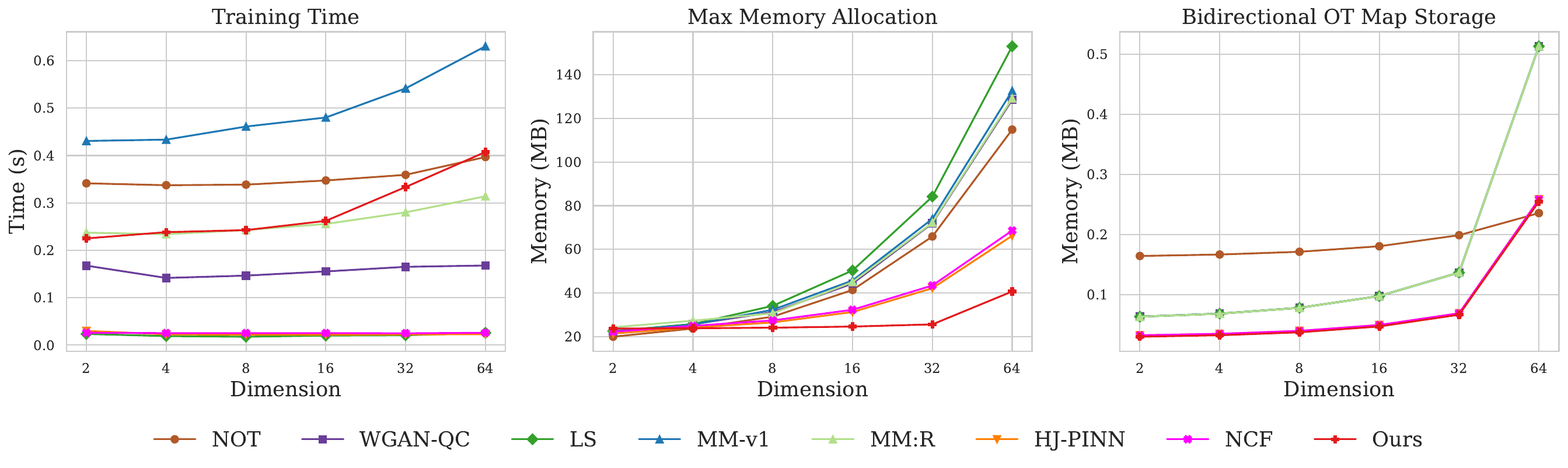}
    \caption{\textit{Computational comparison.} Training time (s/epoch), peak memory (MB) during training, and memory (MB) for storing bidirectional OT maps are reported across models and dimensions.
    }
    \label{fig:landscape_comparison}
    \vspace{-1em}
\end{figure}
Table~\ref{tab:gaussians_uvp} reports the UVP scores across different models and dimensionalities, while Figure~\ref{fig:landscape_comparison} summarizes the corresponding computational costs. 
Compared to all baselines, our method consistently learns substantially more accurate OT maps across all tested dimensions. 
While baseline approaches exhibit a noticeable degradation in accuracy as the dimensionality increases, our model maintains a stable error profile even in high-dimensional settings. 
Notably, a single fixed experimental configuration was used for our model across all dimensions, without any dimension-specific tuning. 
These results highlight the accuracy, robustness, and scalability of our approach for high-dimensional OT.

\subsection{Real-World Physics Data Distributions}\label{sec:uci_physics}

To evaluate the extent to which the proposed method can handle complex and practically relevant distributions, we consider real-world data sets with diverse and heterogeneous characteristics.

\paragraph{Data}
We evaluate on UCI benchmark data sets \cite{uci_repo}, which exhibit diverse dimensions and statistical characteristics, including high-energy physics experiments, household power consumption, and chemical sensor measurements of gas mixtures. Following prior work, we define an OT problem from a standard Gaussian distribution to each empirical data distribution, requiring the model to learn a transport map without paired samples.



\paragraph{Baseline Methods} 
We compare against NCF \cite{park2025neural}, as well as continuous-flow models including FFJORD \cite{grathwohl2018ffjord}, RNODE \cite{finlay2020train}, and OT-Flow \cite{onken2021ot}. These methods provide strong baselines for evaluating both transport quality and computational efficiency.


\begin{table}[t]
\vspace{-1em}
\centering
\small
\setlength{\tabcolsep}{6pt}
\resizebox{0.95\textwidth}{!}{%
\begin{tabular}{l|c|ccccc}
\hline
Dataset & Metric & FFJORD & RNODE & OT-Flow & NCF & Ours \\
\hline

\multirow{2}{*}{Power ($d=6$)} 
& MMD 
& $4.34\!\times\!10^{-5}$ 
& $5.64\!\times\!10^{-5}$ 
& $4.68\!\times\!10^{-5}$ 
& $2.56\!\times\!10^{-4}$
& $\mathbf{1.74\!\times\!10^{-5}}$ \\
& \#Params 
& 43K & 43K & 18K & 8.25K & 8.25K \\
\hline

\multirow{2}{*}{Gas ($d=8$)} 
& MMD 
& $1.02\!\times\!10^{-4}$ 
& $8.03\!\times\!10^{-5}$ 
& $2.47\!\times\!10^{-4}$ 
& $2.24\!\times\!10^{-4}$
& $\mathbf{7.34\!\times\!10^{-5}}$ \\
& \#Params 
& 279K & 279K & 127K & 8.90K & 8.90K \\
\hline

\multirow{2}{*}{HEPMASS ($d=21$)} 
& MMD 
& $1.58\!\times\!10^{-5}$ 
& $1.58\!\times\!10^{-5}$ 
& $1.58\!\times\!10^{-5}$ 
& $6.84\!\times\!10^{-5}$
& $1.79\!\times\!10^{-5}$ \\
& \#Params 
& 547K & 547K & 72K & 13.06K & 13.06K \\
\hline

\multirow{2}{*}{MINIBOONE ($d=43$)} 
& MMD 
& \textbf{$2.84\!\times\!10^{-4}$}
& \textbf{$2.84\!\times\!10^{-4}$} 
& \textbf{$2.84\!\times\!10^{-4}$} 
& $2.78\!\times\!10^{-3}$
& $1.19\!\times\!10^{-3}$ \\
& \#Params 
& 821K & 821K & 78K & 29.84K & 29.84K \\
\hline

\multirow{2}{*}{BSDS300 ($d=63$)} 
& MMD 
& $6.52\!\times\!10^{-3}$ 
& $1.64\!\times\!10^{-2}$ 
& $4.24\!\times\!10^{-4}$ 
& $6.28\!\times\!10^{-4}$
& $\mathbf{3.30\!\times\!10^{-4}}$ \\
& \#Params 
& 6.7M & 6.7M & 297K & 156K & 156K \\
\hline

\end{tabular}}
\caption{Comparison of distribution matching performance (MMD) and model size (\#parameters) across real-world datasets. Lower MMD indicates better distribution matching.}
\label{tab:uci_physics}
\vspace{-2em}
\end{table}

\paragraph{Evaluation Metrics} 
For evaluation, we adopt the same setup as \cite{onken2021ot}, using the unbiased \emph{Maximum Mean Discrepancy (MMD)} estimator with Gaussian kernels to measure the distance between generated and target distributions. Specifically, we use a kernel of the form
\[
k(x,x') = \sum_{j=1}^K \exp(-\alpha_j \|x - x'\|^2),
\]
where the bandwidth parameters $\{\alpha_j\}$ are selected as in prior work. Given two sets of samples $\{x_i\}_{i=1}^{n_1}$ and $\{y_j\}_{j=1}^{n_2}$, the estimator is given by
\[
\mathrm{MMD}^2 = \frac{1}{n_1(n_1-1)} \sum_{i \neq j} k(x_i,x_j)
+ \frac{1}{n_2(n_2-1)} \sum_{i \neq j} k(y_i,y_j)
- \frac{2}{n_1 n_2} \sum_{i,j} k(x_i,y_j).
\]

To assess computational efficiency, we additionally measure the number of trainable parameters in each model.

\paragraph{Results}
Results are summarized in Table~\ref{tab:uci_physics}, with additional visualizations for GAS, HEPMASS, and MINIBOONE in Figures~\ref{fig:gas}, \ref{fig:hepmass}, and \ref{fig:miniboone}. Since the datasets are high-dimensional, we visualize two-dimensional projections along informative coordinate axes.

The proposed method accurately captures the structure of complex target distributions and achieves comparable or better performance than continuous flow-based baselines, despite using smaller network architectures. In contrast to flow-based models, which require continuous-time ODE integration at inference, our method directly solves the OT problem and enables more efficient sampling. Compared to NCF, our approach shows consistently stronger performance on challenging high-dimensional distributions, indicating improved scalability.

\begin{figure}[t]
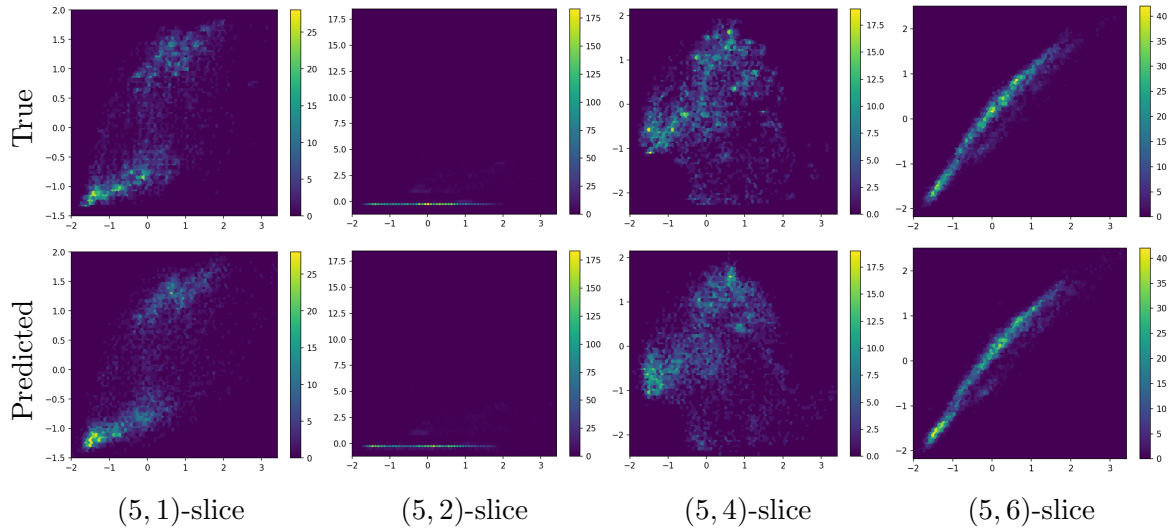

\centering

\begin{tikzpicture}

\def\xgap{3.75}   
\def\ygap{-3.2}  

\def\imgw{0.24\linewidth} 

\foreach \j/\col in {1/0,2/1,4/2,6/3} {

    \node[anchor=center] (T\col)
    at (\col*\xgap, 0)
    {\includegraphics[width=\imgw]{Figures/gas/5_\j_real.png}};

    \node[anchor=center] (P\col)
    at (\col*\xgap, \ygap)
    {\includegraphics[width=\imgw]{Figures/gas/5_\j_pred.png}};

    \node[align=center] at (\col*\xgap, -5.2)
    {$(5,\j)$-slice};
}

\node[rotate=90, align=center]
at (-2.1, 0.0)
{True};

\node[rotate=90, align=center]
at (-2.1, -3.2)
{Predicted};

\end{tikzpicture}

\caption{
Each column shows a different slice of the UCI Physics Gas dataset, comparing real (top) and predicted (bottom) samples.
}\label{fig:gas}
\vspace{-1em}
\end{figure}

\begin{figure}[t]
\centering

\begin{tikzpicture}

\def\xgap{3.75}   
\def\ygap{-3.2}  

\def\imgw{0.24\linewidth} 

\foreach \j/\col in {3/0,6/1,18/2,19/3} {

    \node[anchor=center] (T\col)
    at (\col*\xgap, 0)
    {\includegraphics[width=\imgw]{Figures/hepmass/\j_21_real.png}};

    \node[anchor=center] (P\col)
    at (\col*\xgap, \ygap)
    {\includegraphics[width=\imgw]{Figures/hepmass/\j_21_pred.png}};

    \node[align=center] at (\col*\xgap, -5.2)
    {$(\j, 21)$-slice};
}

\node[rotate=90, align=center]
at (-2.1, 0.0)
{True};

\node[rotate=90, align=center]
at (-2.1, -3.2)
{Predicted};

\end{tikzpicture}

\caption{
Each column shows a different slice of the UCI Physics Hepmass dataset, comparing real (top) and predicted (bottom) samples.
}\label{fig:hepmass}
\vspace{-1em}
\end{figure}

\begin{figure}[t]
    \centering
    \includegraphics[width=1.0\textwidth]{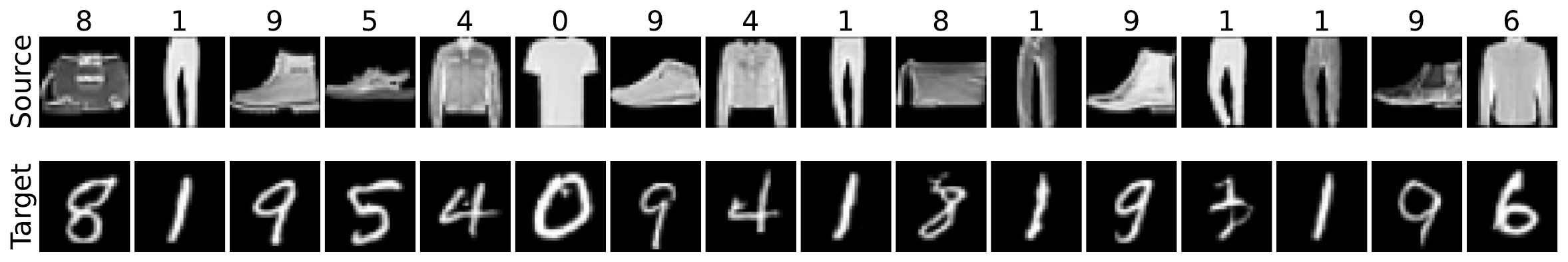}
    \caption{Conditional OT results for Fashion-MNIST $\rightarrow$ MNIST. Top: input images; bottom: transported outputs. Each column corresponds to a class label.}
    \label{fig:fmnist_to_mnist}
    \vspace{-1em}
\end{figure}

\begin{figure}[t]
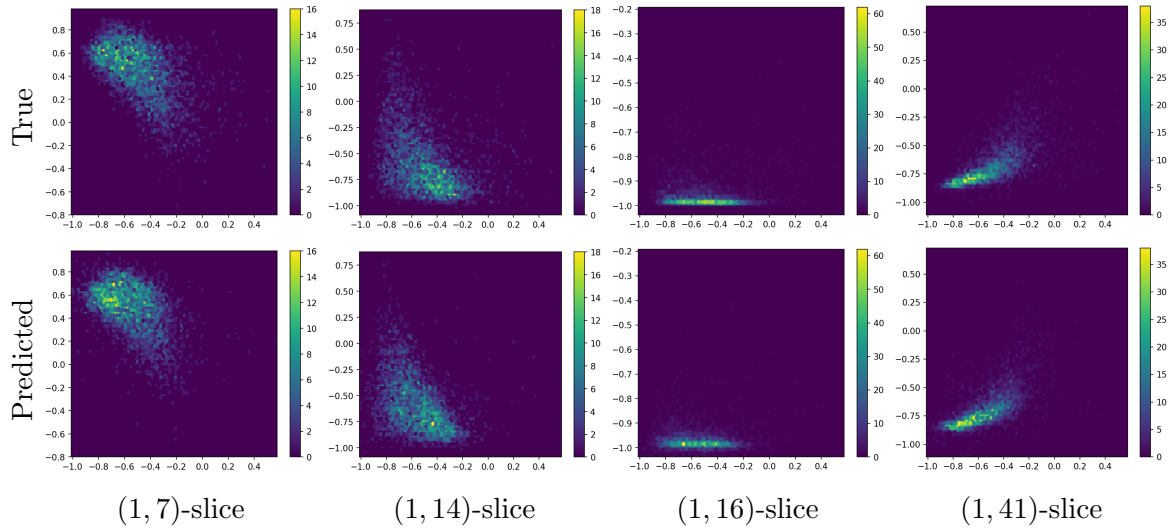

\centering

\begin{tikzpicture}

\def\xgap{3.75}   
\def\ygap{-3.2}  

\def\imgw{0.24\linewidth} 

\foreach \j/\col in {6/0,13/1,15/2,40/3} {

    \node[anchor=center] (T\col)
    at (\col*\xgap, 0)
    {\includegraphics[width=\imgw]{Figures/Miniboone/0_\j_real.png}};

    \node[anchor=center] (P\col)
    at (\col*\xgap, \ygap)
    {\includegraphics[width=\imgw]{Figures/Miniboone/0_\j_pred.png}};

    \pgfmathtruncatemacro{\jj}{\j + 1}
    \node[align=center] at (\col*\xgap, -5.2)
    {$(1,\jj)$-slice};
}

\node[rotate=90, align=center]
at (-2.1, 0.0)
{True};

\node[rotate=90, align=center]
at (-2.1, -3.2)
{Predicted};

\end{tikzpicture}

\caption{
Each column shows a different slice of the UCI Physics Miniboone dataset, comparing real (top) and predicted (bottom) samples.
}\label{fig:miniboone}
\vspace{-1em}
\end{figure}

\subsection{2D Class-Conditional Transport}\label{sec:ccot_2d}
We present experimental results on a two-dimensional (2D) synthetic dataset consisting of class-labeled samples, designed to evaluate CCOT. 

\paragraph{Data} The dataset is constructed as a mixture of Gaussian components, where each component corresponds to a distinct class label. This setting allows us to assess whether the learned transport respects both inter-class alignment and intra-class structure preservation.

\paragraph{Results} Figure~\ref{fig:class_2d} illustrates the learned transport behavior across three 2D Gaussian mixture datasets. Each data point is associated with a class label, and transport is performed in a class-conditional manner. The results show that our method successfully aligns corresponding classes between source and target distributions while maintaining clear separation between different classes.

In addition to global distribution matching, the learned transport preserves the local geometry within each class, indicating that the model captures both class-level correspondence and fine-grained structural consistency. This suggests that the proposed formulation can naturally extend to conditional settings without requiring architectural modifications or additional networks.

\begin{figure}
    \centering
    \begin{tikzpicture}[every node/.style={font=\small}]

        \node[anchor=north west, inner sep=0] (data1)
        at (0, 0) {\includegraphics[width=0.25\linewidth]{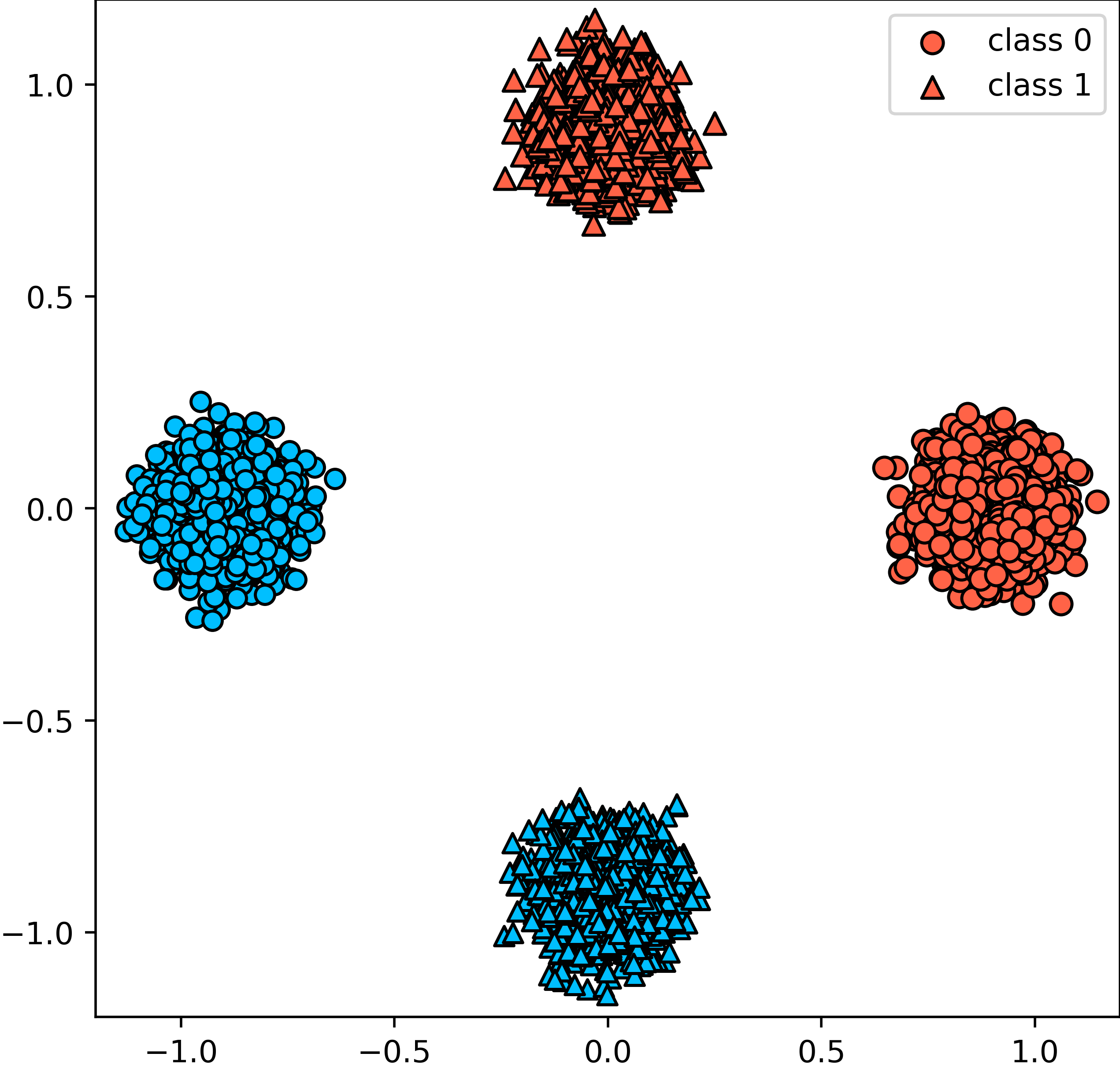}};
        \node[anchor=north west, inner sep=0] (fwd1)
        at (0.30\linewidth, 0) {\includegraphics[width=0.25\linewidth]{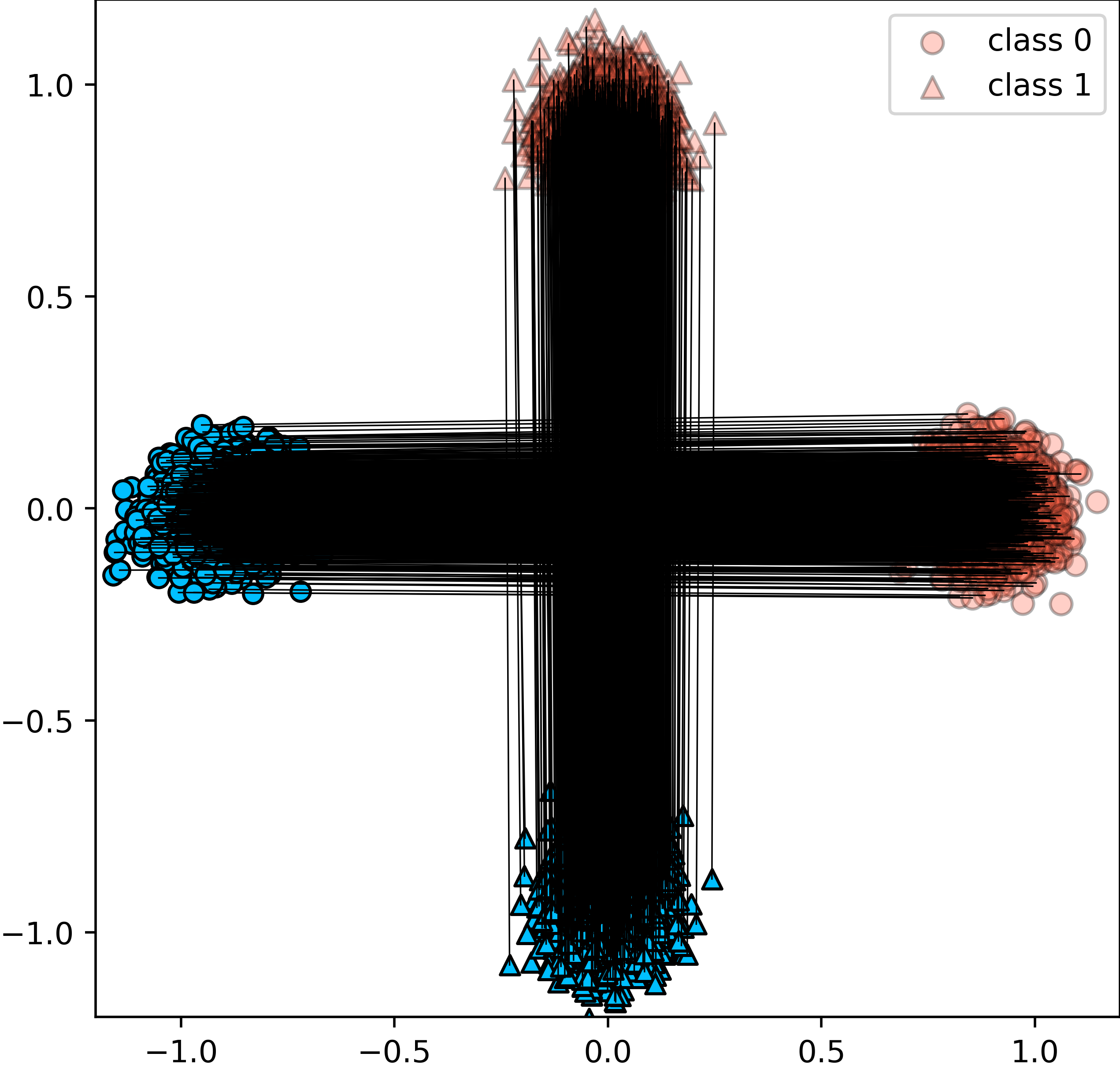}};
        \node[anchor=north west, inner sep=0] (bwd1)
        at (0.60\linewidth, 0) {\includegraphics[width=0.25\linewidth]{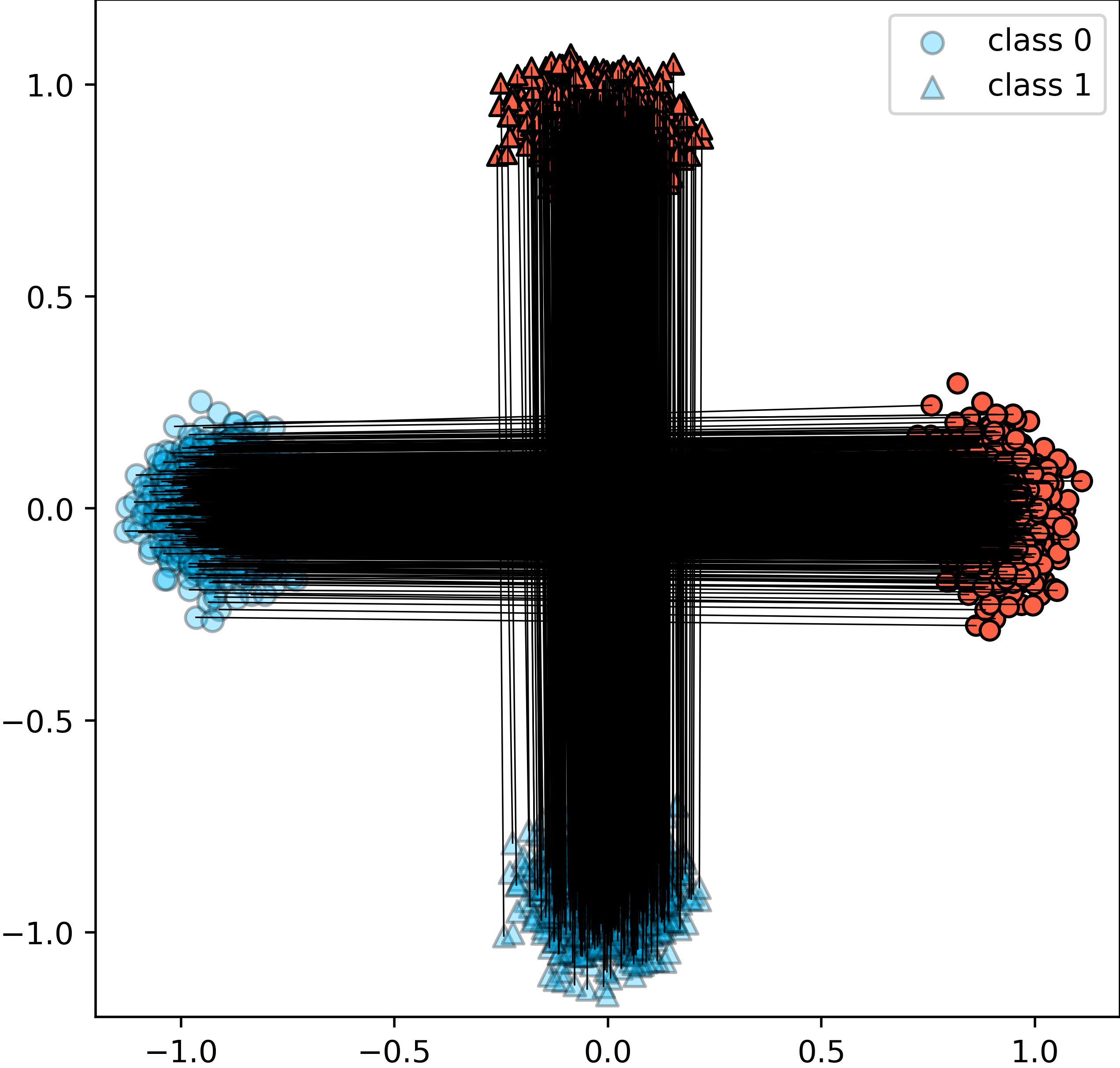}};

        \node[anchor=north west, inner sep=0] (data2)
        at (0, -4.0) {\includegraphics[width=0.25\linewidth]{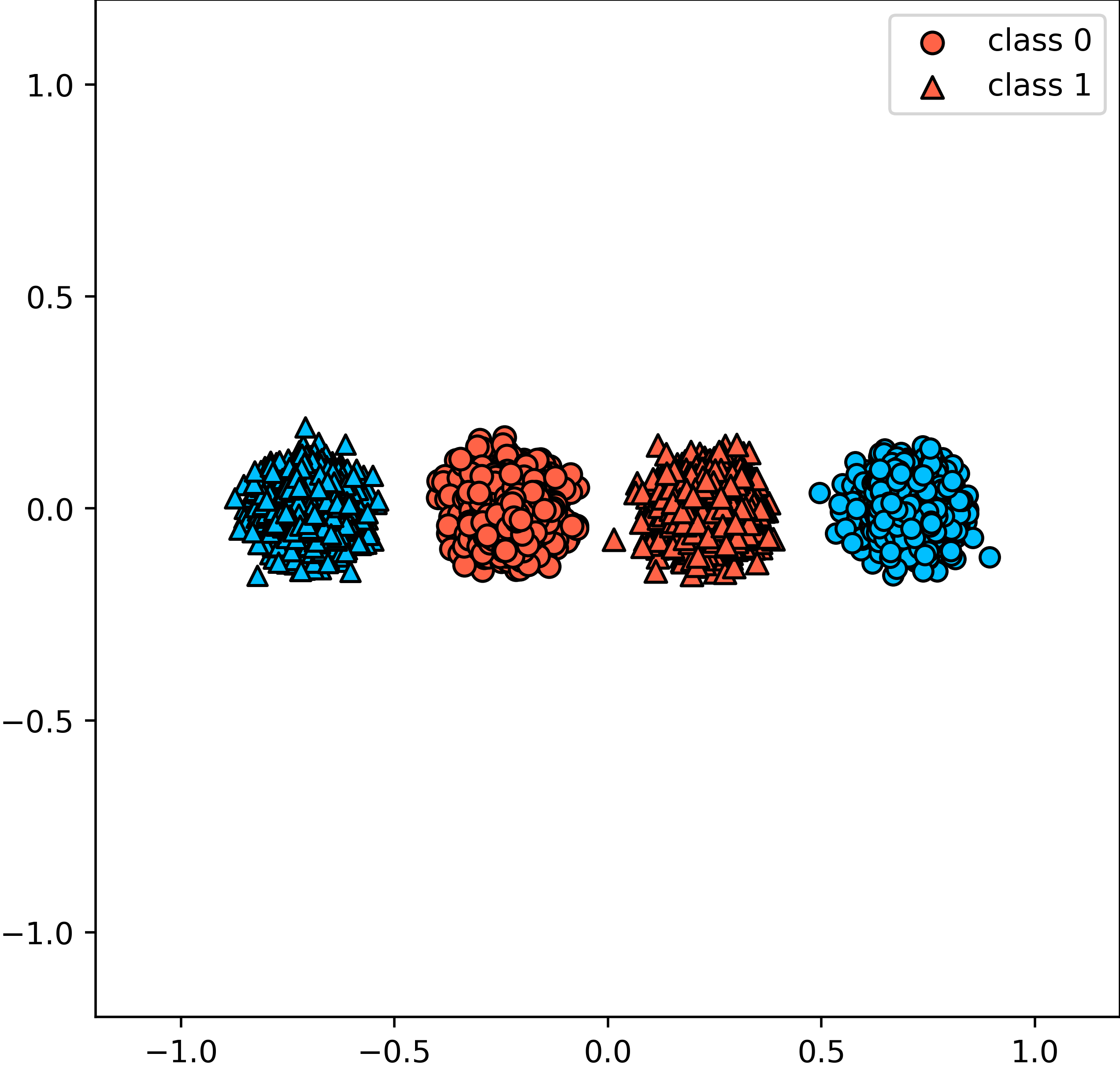}};
        \node[anchor=north west, inner sep=0] (fwd2)
        at (0.30\linewidth, -4.0) {\includegraphics[width=0.25\linewidth]{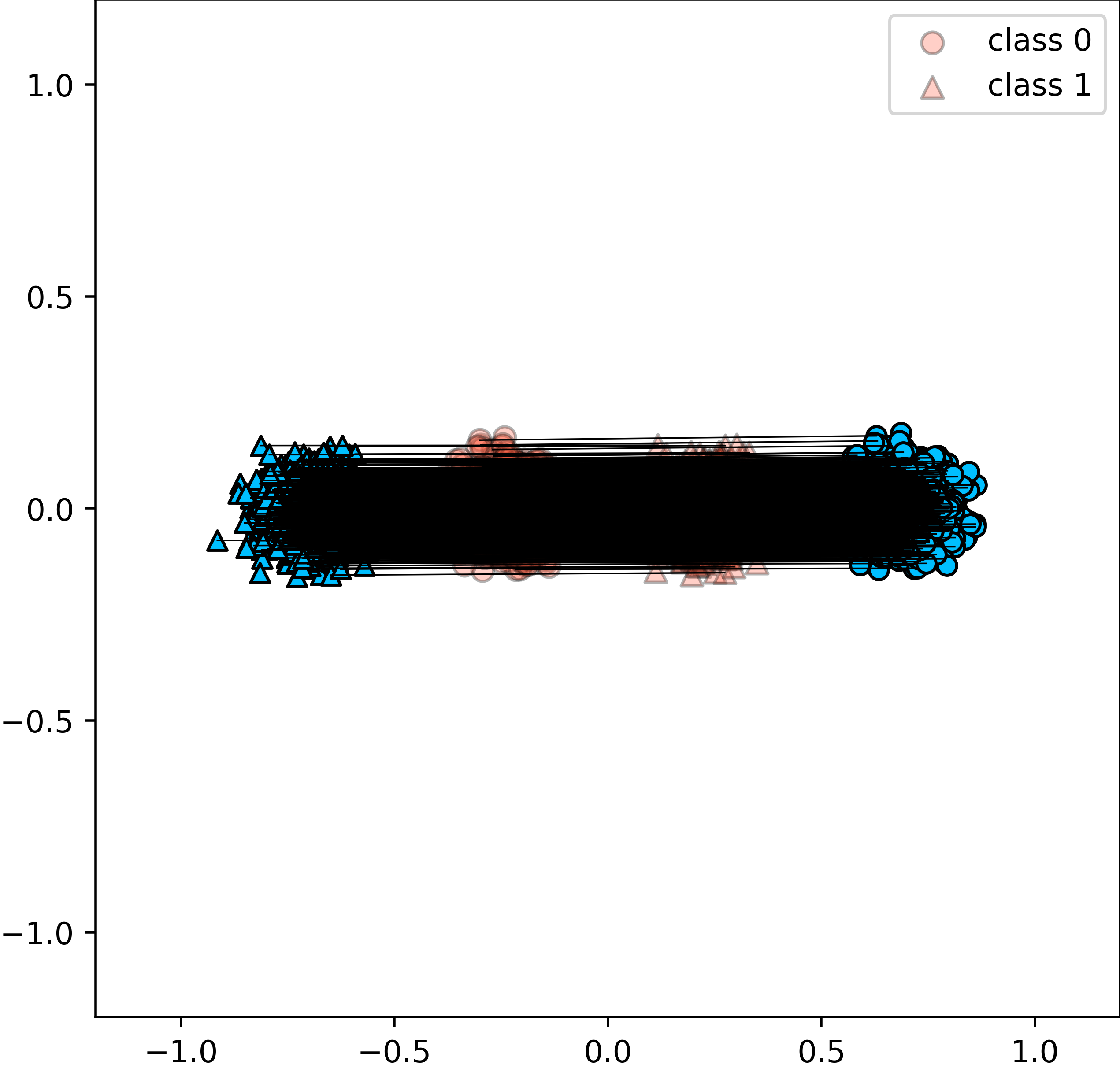}};
        \node[anchor=north west, inner sep=0] (bwd2)
        at (0.60\linewidth, -4.0) {\includegraphics[width=0.25\linewidth]{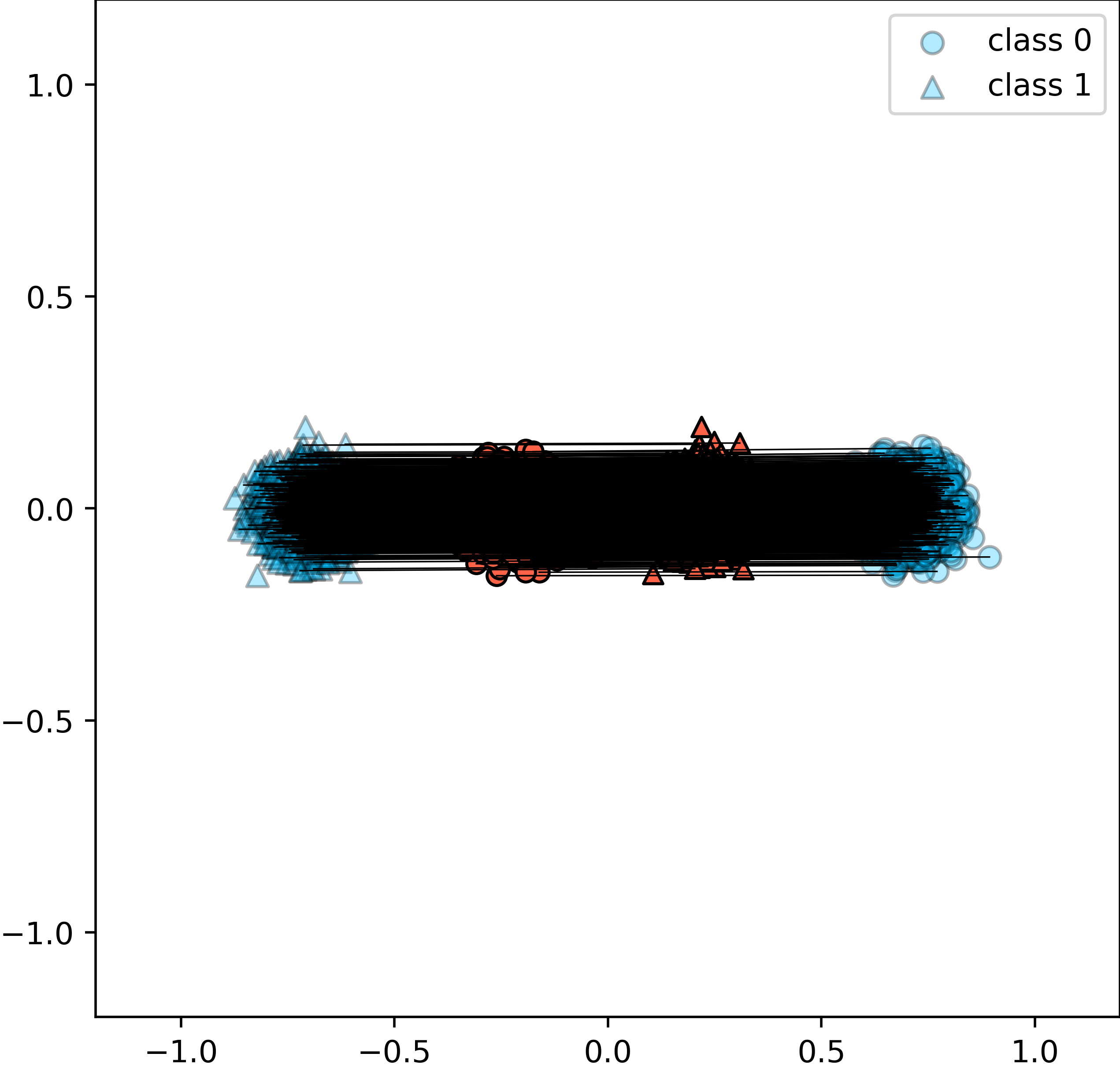}};

        \node[anchor=north west, inner sep=0] (data3)
        at (0, -8.0) {\includegraphics[width=0.25\linewidth]{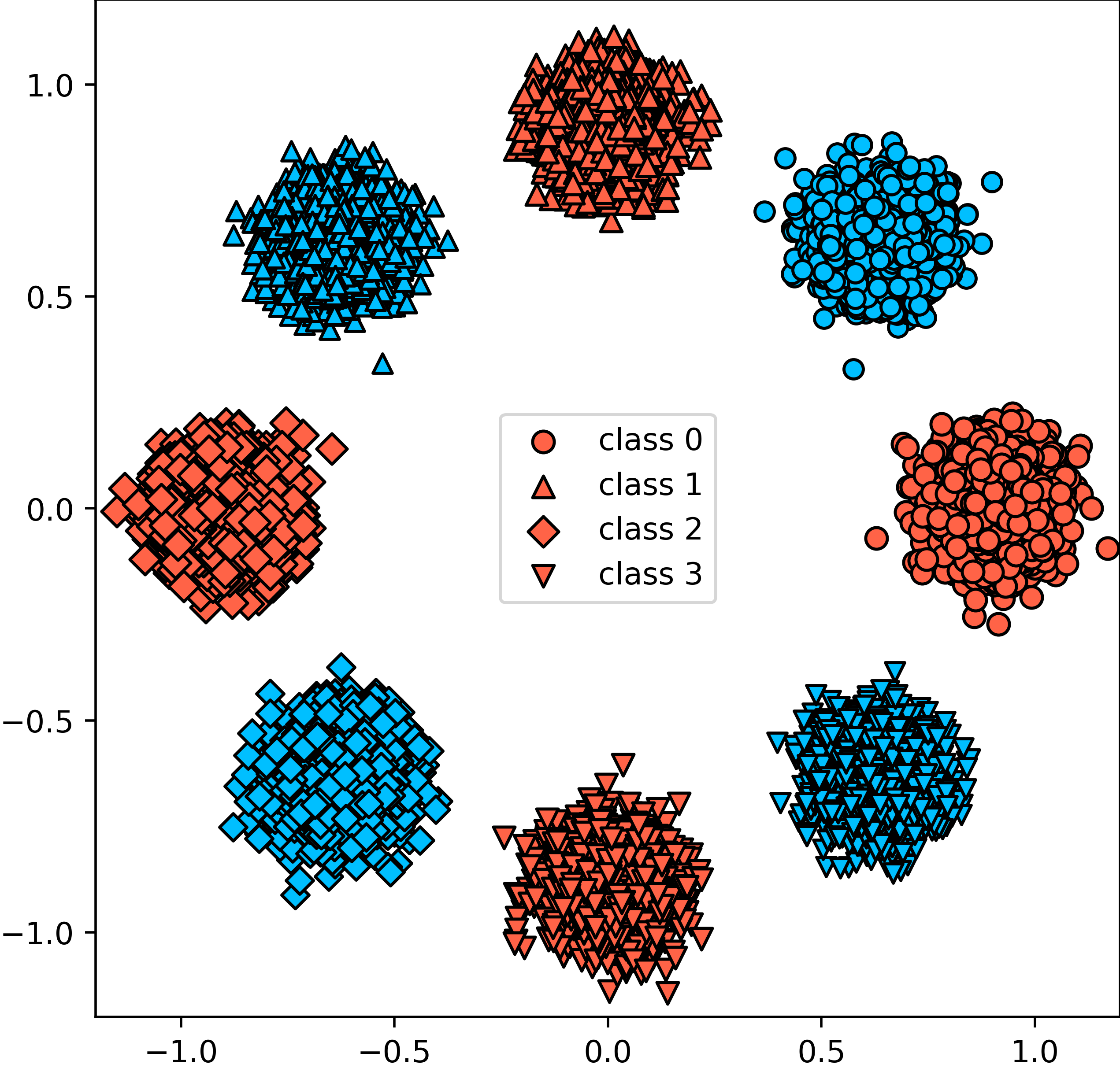}};
        \node[anchor=north west, inner sep=0] (fwd3)
        at (0.30\linewidth, -8.0) {\includegraphics[width=0.25\linewidth]{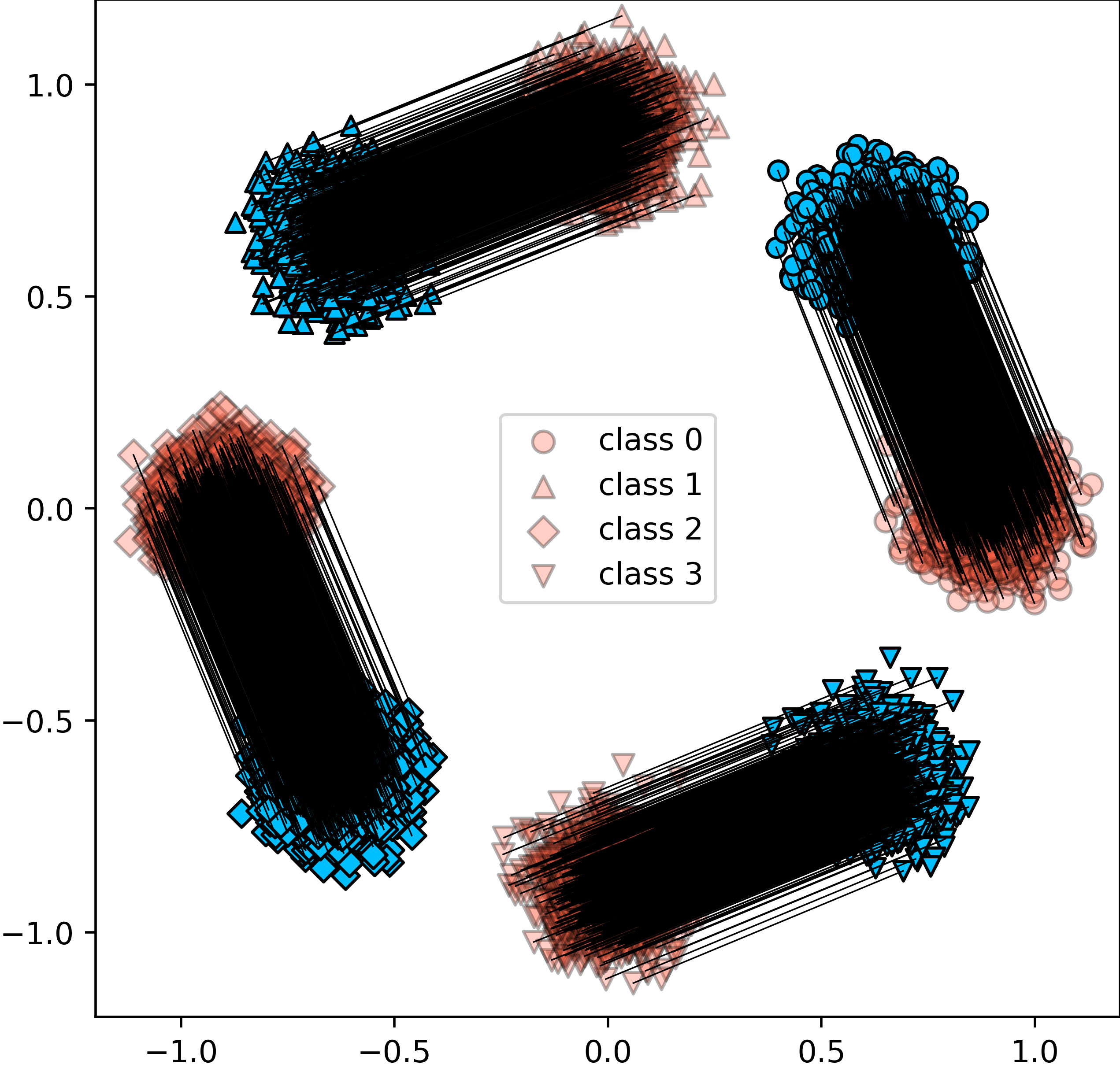}};
        \node[anchor=north west, inner sep=0] (bwd3)
        at (0.60\linewidth, -8.0) {\includegraphics[width=0.25\linewidth]{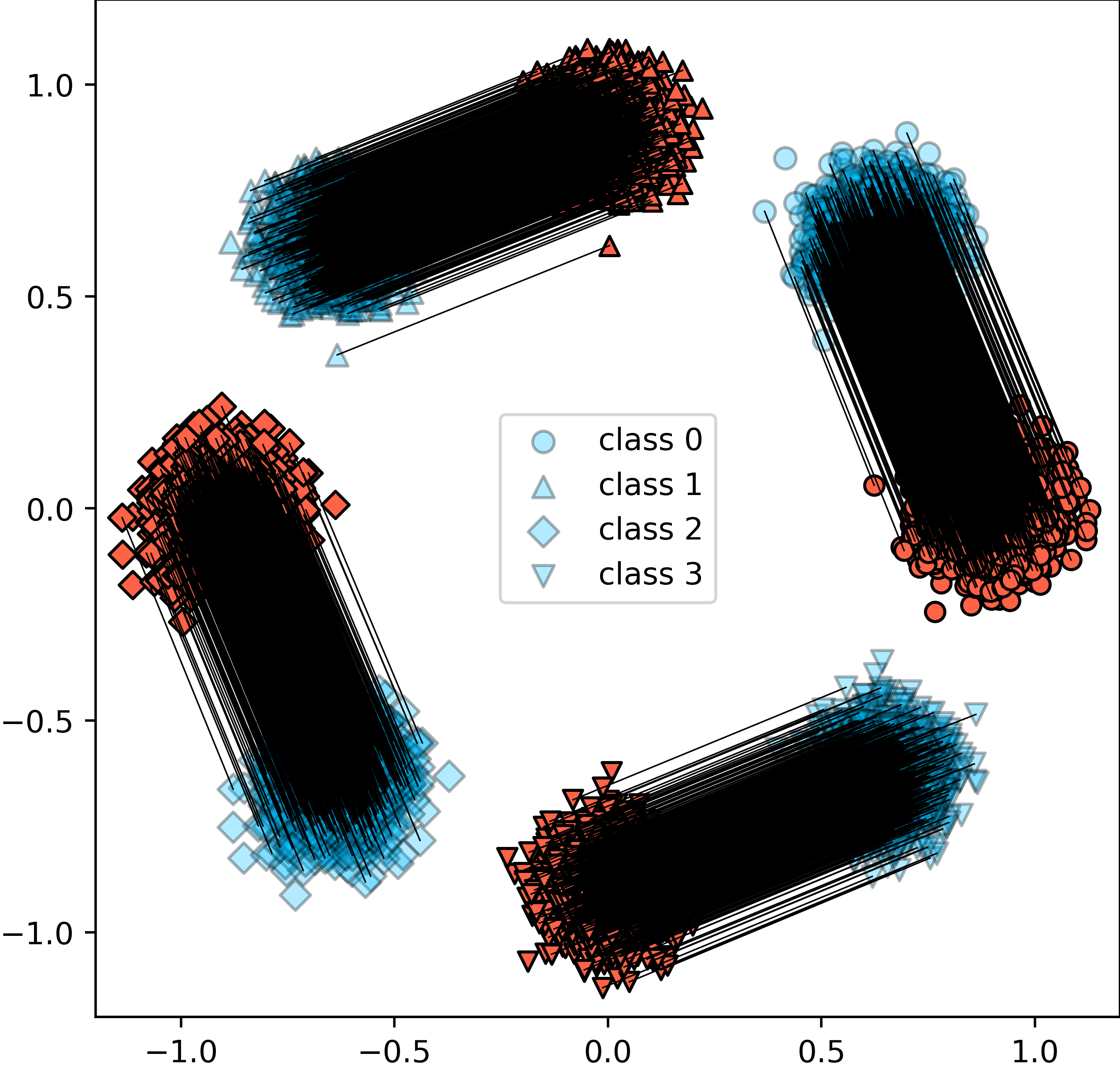}};

        \node[anchor=south] at ($(data1.north)+(0,0.3)$) {Data};
        \node[anchor=south] at ($(fwd1.north)+(0,0.3)$) {Forward OT};
        \node[anchor=south] at ($(bwd1.north)+(0,0.3)$) {Backward OT};

        \node[rotate=90, anchor=south] at ($(data1.west)+(-1.0,-0)$)
            {\textsf{Cross Ring}};
        \node[rotate=90, anchor=south] at ($(data2.west)+(-1.0,-0)$)
            {\textsf{Horizontal Swapped}};
        \node[rotate=90, anchor=south] at ($(data3.west)+(-1.0,-0)$)
            {\textsf{Four-Modes}};

    \end{tikzpicture}
    \caption{
    \textit{CCOT on Gaussian mixtures.}
    Each row corresponds to a different class-structured problem.
    The first column shows the empirical data distributions,
    the second column visualizes the learned forward transport map,
    and the third column shows the learned backward transport. Colors indicate different distributions, while marker shapes differentiate individual classes within each distribution.
    }
    \label{fig:class_2d}
    \vspace{-1em}
\end{figure}

\subsection{Class-Conditional Image Translation}\label{sec:mnist_fmnist}
We further evaluate our method on CCOT, where the goal is to transport samples between image distributions while preserving semantic class information. This setting tests whether the learned transport map achieves both distributional alignment and class consistency.

\paragraph{Data}
We use MNIST \cite{lecun1998mnist} and Fashion-MNIST (FMNIST) \cite{xiao2017fashion}, each consisting of 10 semantic classes. Since OT in pixel space is ill-posed for image data \cite{pope2021intrinsic}, we perform transport in a shared VAE latent space of dimension 15 and decode transported samples back to the image domain.




\paragraph{Baseline Methods}
Following prior work \cite{asadulaev2022neural, park2025neural}, we compare against representative image translation and OT methods. These include pixel-level translation models, MUNIT \cite{huang2018multimodal} and AugCycleGAN \cite{almahairi2018augmented, zhu2017unpaired}; OTDD flow \cite{alvarez2021dataset, alvarez2021optimizing}, which preserves class structure through gradient flows; and discrete OT (DOT) based on Sinkhorn optimization with Laplacian regularization \cite{cuturi2013sinkhorn, courty2016optimal}. We also evaluate neural OT methods \cite{korotin2022neural, fan2023neural, asadulaev2022neural} under quadratic (W2), weak quadratic (W2,$\gamma$), and class-conditional (FG) costs, as well as NCF \cite{park2025neural}, which solves class-conditional transport via Hamilton--Jacobi dynamics.




\paragraph{Evaluation Metrics}
We evaluate image quality using the \emph{Fréchet Inception Distance (FID)}, which measures the discrepancy between feature distributions of real and generated samples:
[
$\mathrm{FID} = |\mu_r - \mu_g|^2 + \mathrm{Tr}\left(\Sigma_r + \Sigma_g - 2(\Sigma_r \Sigma_g)^{1/2}\right)$.
]

To assess class preservation, we follow \cite{asadulaev2022neural} and report \emph{classification accuracy} using a pretrained ResNet-18 classifier \cite{He_2016_CVPR}. Accuracy is computed as the fraction of transported samples whose predicted label matches the source class label.



\paragraph{Results}
FID scores and classification accuracies are summarized in Tables~\ref{tab:fid_scores} and \ref{tab:classification_accuracy}, respectively. Moreover, qualitative results are presented in Figure~\ref{fig:fmnist_to_mnist}. As shown in Table~\ref{tab:fid_scores}, our method achieves the lowest FID among all competing baselines. In particular, it significantly outperforms both pixel-level generative models (AugCycleGAN and MUNIT) and classical OT approaches (W2, W2-$\gamma$, DOT, and OTDD flow). Compared to neural OT methods such as GNOT and NCF, our approach consistently yields improved perceptual quality.
As shown in Table~\ref{tab:classification_accuracy}, our method achieves the highest classification accuracy among all baselines, indicating that it not only aligns marginal distributions but also preserves class-level structure more effectively than competing approaches.

This is further illustrated in Figure~\ref{fig:fmnist_to_mnist}, where the proposed method successfully performs class-consistent transport, aligning samples according to their respective class labels while preserving semantic structure across domains. These results suggest that the proposed single-network minimization framework effectively captures both distributional alignment and class-conditional structure in the latent space, leading to more faithful class-conditional transport compared to adversarial and flow-based neural OT methods.

\begin{table}[t]
\centering
\footnotesize
\setlength{\tabcolsep}{3pt}
\renewcommand{\arraystretch}{0.95}
\caption{Accuracy $\uparrow$ of transported samples.}
\resizebox{\linewidth}{!}{
\begin{tabular}{lcccc|ccc|ccc}
\toprule
& \multicolumn{2}{c|}{Pixel-level} 
& \multicolumn{2}{c|}{Discrete / OTDD}
& \multicolumn{3}{c|}{Neural OT}
& \multicolumn{3}{c}{Ours / Neural Flow} \\
\cmidrule(lr){2-3} \cmidrule(lr){4-5} \cmidrule(lr){6-8} \cmidrule(lr){9-11}
Datasets 
& MUNIT & AugCG 
& OTDD & SinkLpL1 
& W2 & W2,$\gamma$ & FG 
& NCF & FG(no $z$) & \textbf{Ours} \\
\midrule
FMNIST $\rightarrow$ MNIST 
& 8.93 & 12.03 
& 10.28 & 10.67 
& 10.96 & 8.02 & 83.72
& 83.42 & 82.79 & \textbf{96.00} \\
\bottomrule
\end{tabular}
}
\label{tab:classification_accuracy}
\end{table}

\begin{table}[t]
\centering
\footnotesize
\setlength{\tabcolsep}{3pt}
\renewcommand{\arraystretch}{0.95}
\caption{FID $\downarrow$ of transported samples.}
\resizebox{\linewidth}{!}{
\begin{tabular}{lcccc|ccc|ccc}
\toprule
& \multicolumn{2}{c|}{Pixel-level} 
& \multicolumn{2}{c|}{Discrete / OTDD}
& \multicolumn{3}{c|}{Neural OT}
& \multicolumn{3}{c}{Ours / Neural Flow} \\
\cmidrule(lr){2-3} \cmidrule(lr){4-5} \cmidrule(lr){6-8} \cmidrule(lr){9-11}
Datasets 
& MUNIT & AugCG 
& OTDD & SinkLpL1 
& W2 & W2,$\gamma$ & FG 
& NCF & FG(no $z$) & \textbf{Ours} \\
\midrule
FMNIST $\rightarrow$ MNIST 
& 7.91 & 26.35 
& $>$100 & $>$100 
& 7.51 & 7.02 & 5.26 
& 18.27 & 7.14 & \textbf{4.57} \\
\bottomrule
\end{tabular}
}
\label{tab:fid_scores}
\vspace{-1em}
\end{table}

\subsection{Ablation Studies}\label{sec:ablation}
\subsubsection{Ablation Study on Fixed-Point Tolerance}
In practical implementations, the proximal subproblem in \eqref{eq:prox_def} is solved using an iterative fixed-point scheme, which is terminated once a prescribed tolerance $\epsilon_p$ is reached. While the theoretical analysis assumes exact computation of $y_\theta^\star(z)$, training is carried out using an approximate solution $\tilde{y}_\theta(z)$. It is therefore important to quantify the effect of the stopping tolerance on both accuracy and computational efficiency.

To this end, we perform an ablation study on the 32-dimensional Gaussian OT task presented in Section \ref{sec:high_dim_gaussians_rewrite}, systematically varying the tolerance $\epsilon_p$. For each configuration, we report the forward and backward transport errors (UVP), along with the average training time per epoch.

\begin{table}[h]
\caption{Sensitivity of transport accuracy and computational cost to the fixed-point stopping tolerance $\epsilon_p$.}
\label{tab:tolerance_ablation}
\centering
\resizebox{\columnwidth}{!}{
\begin{tabular}{c|cccccccc}
\hline
 & $10^{-4}$ & $5\times10^{-4}$ & $10^{-3}$ & $5\times10^{-3}$ & $10^{-2}$ & $5\times10^{-2}$ & $10^{-1}$ & $5\times10^{-1}$ \\
\hline
Forward UVP
& $4.909\times10^{-2}$
& $5.125\times10^{-2}$
& $5.435\times10^{-2}$
& $5.492\times10^{-2}$
& $5.449\times10^{-2}$
& $5.622\times10^{-2}$
& $6.255\times10^{-2}$
& $2.229\times10^{-1}$ \\
Backward UVP
& $4.968\times10^{-2}$
& $5.906\times10^{-2}$
& $5.777\times10^{-2}$
& $5.861\times10^{-2}$
& $5.925\times10^{-2}$
& $6.051\times10^{-2}$
& $6.557\times10^{-2}$
& $1.724\times10^{-1}$ \\
Time (s)
& $0.986$
& $0.917$
& $0.856$
& $0.739$
& $0.640$
& $0.516$
& $0.479$
& $0.277$ \\
\hline
\end{tabular}}
\vspace{-1em}
\end{table}


The results exhibit a consistent trade-off between computational cost and estimation accuracy, while simultaneously demonstrating that the method remains robust with respect to the choice of fixed-point tolerance. As predicted by the theory, increasing the tolerance $\epsilon_p$ introduces a discrepancy between the approximate solution $\tilde{y}_\theta(z)$ and the exact minimizer $y_\theta^\star(z)$, which leads to some degradation in accuracy. At the same time, a larger tolerance reduces the number of iterations required for convergence of the fixed-point procedure, thereby improving computational efficiency.

Notably, across the range $\epsilon_p \in [10^{-5}, 10^{-2}]$, the degradation in transport accuracy is marginal, indicating that the method is not highly sensitive to moderate inexactness in the proximal computation. This suggests that the learned transport maps are relatively insensitive to approximation errors in the fixed-point solution within this regime. This observation is consistent with the analysis in Section~\ref{lemma:approx_fixed_point}, which guarantees that sufficiently accurate approximate fixed points remain close to the true minimizer.

From a practical perspective, this empirical behavior indicates that, within a reasonable tolerance regime, the approximate fixed-point $\tilde{y}_\theta(z)$ provides a sufficiently accurate surrogate for $y_\theta^\star(z)$ and does not significantly distort the gradient signal used during training. In contrast, when the tolerance becomes excessively large (e.g., $\epsilon_p \geq 10^{-1}$), the approximation error becomes non-negligible and leads to a clear deterioration in performance.

Overall, these findings suggest that solving the inner optimization problem to very high precision is unnecessary in practice. Instead, moderately accurate fixed-point solutions achieve a favorable balance, providing substantial computational savings while maintaining transport accuracy.

\subsubsection{Ablation Study on Potential Parameterization}
To study the effect of potential parameterization, we compare two ICNN variants: (i) a \textbf{Convex Network}, where convexity is enforced via weight projection, and (ii) a \textbf{Nonconvex Network}, which uses the same architecture without convexity constraints. We further evaluate three activation functions: Tanh, SoftPlus, and CeLU. All experiments are conducted in the 32-dimensional setting. Results are reported in Table~\ref{tab:activation_ablation}, and the fixed-point iteration depth and residual
$|y^{k+1}-z+\nabla g(y^{k+1})|_\infty$
are shown in Figure~\ref{fig:network_ablation}.

\begin{table}[t]
\centering
\footnotesize
\setlength{\tabcolsep}{4pt}
\renewcommand{\arraystretch}{0.8}
\caption{Comparison of network parameterization.}\label{tab:ablation_network}
\resizebox{0.6\columnwidth}{!}{
\begin{tabular}{lcc|cc}
\toprule
 & \multicolumn{2}{c}{\textbf{Convex}} 
 & \multicolumn{2}{c}{\textbf{Non-convex}} \\
\cmidrule(lr){2-3} \cmidrule(lr){4-5}
\textbf{Activation} 
& Fwd UVP & Bwd UVP 
& Fwd UVP & Bwd UVP \\
\midrule
Tanh 
& $7.173 \times 10^{-1}$ & $7.565 \times 10^{-1}$ 
& $1.730 \times 10^{-1}$ & $2.184 \times 10^{-1}$ \\

SoftPlus 
& $6.857 \times 10^{0}$ & $5.945 \times 10^{0}$ 
& $6.996 \times 10^{-2}$ & $7.551 \times 10^{-2}$ \\

CeLU 
& $6.826 \times 10^{0}$ & $5.794 \times 10^{0}$ 
& $5.4 \times 10^{-2}$ & $5.67 \times 10^{-2}$ \\
\bottomrule
\end{tabular}}
\label{tab:activation_ablation}
\end{table}
\begin{figure}[t]
    \centering
    \includegraphics[width=0.99\linewidth]{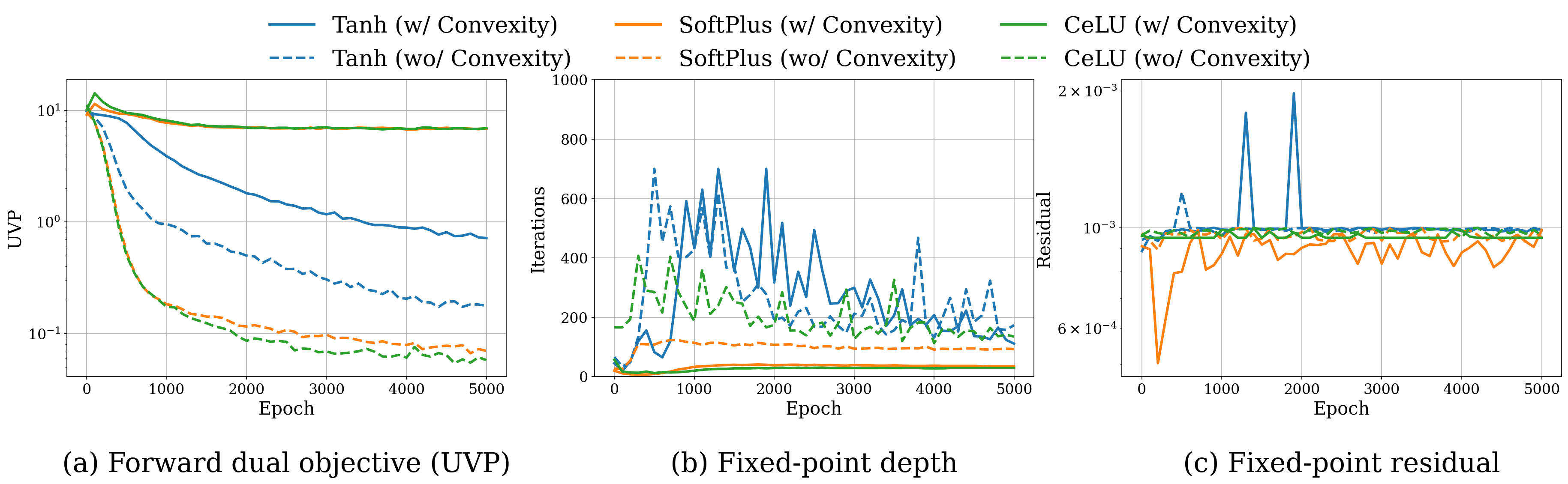}
    \caption{Ablation study on network parameterizations. We compare convex and nonconvex network with different activation functions across (a) OT map accuracy (UVP), (b) the number of fixed-point iteration to converge, and (c) convergence residual.}
    \label{fig:network_ablation}
    \vspace{-1em}
\end{figure}

\paragraph{Trade-off Perspective.}
Table~\ref{tab:activation_ablation} and Figure~\ref{fig:network_ablation} reveal a trade-off between the expressiveness of the learned potential and the stability of the induced fixed-point iterations. Convexified ICNNs yield the most stable and fastest convergence due to their monotone gradient fields, but their constrained parameterization limits flexibility and leads to inferior OT objective values. While convexity is theoretically well-motivated by the Kantorovich formulation, strong convexification restricts the effective parameter space and introduces optimization bias. 

In contrast, nonconvex ICNNs achieve the best overall performance, providing sufficient expressive power while maintaining stable fixed-point dynamics. This suggests that explicit convexification is not necessary in practice and may overly restrict the learned transport geometry.

Activation functions exhibit a similar effect. Tanh consistently performs worse, likely due to gradient saturation, which degrades both optimization and fixed-point convergence. In contrast, SoftPlus and CeLU provide smoother gradient flow and better-conditioned optimization, resulting in improved convergence and transport performance.

\section{Conclusion}\label{sec:conclusion}
We studied an implicit OT formulation based on a single neural potential, where both forward and backward transport maps are recovered through a unified minimization problem. By exploiting the proximal structure of the dual formulation, the objective reduces to a fixed-point problem whose gradient can be computed without implicit differentiation. We establish stochastic gradient descent convergence under inexact inner fixed-point iterations, which arise in practice due to finite computational budgets.
This yields a simple and stable training procedure that avoids saddle-point optimization, auxiliary networks, and unrolled dynamics, while maintaining strong empirical performance in high-dimensional and class-conditional settings.

Future work includes extending the evaluation to more diverse and complex datasets, as well as establishing stronger theoretical guarantees, in particular convergence to the true OT map under practical neural network parameterizations. Furthermore, generalizing the proposed framework beyond the quadratic cost to more general cost functions remains an important direction for broadening its applicability.

\appendix
\section{Proofs}\label{appen:proof}

\subsection{Proof of Theorem ~\ref{theorem:optimal_transport_map}}
\begin{proof}
Because $\mu, \nu$ are probability measures on $\mathbb{R}^d$ with finite second moments, and $\mu$ is absolutely continuous with respect to the Lebesgue measure on $\mathbb{R}^d$, all of the assumptions of Brenier's Theorem~\cite[Theorem 2.5.10]{figalli_glaudo_OT} are true. Then by ~\cite[Corollary 2.5.12]{figalli_glaudo_OT} $\exists$ a unique optimal transport map $T: \mathbb{R}^d \rightarrow \mathbb{R}^d$ such that $T_{\#}\mu = \nu$ and $T = \nabla u$ for some convex $u: \mathbb{R}^d \rightarrow \mathbb{R} \cup \{+\infty\}$. \newline \newline
By~\cite[Theorem 10.28]{villani2008optimal}, $T$ satisfies $\nabla g(x) + \nabla_x c(x,T(x)) = 0$ $\mu$-almost everywhere for some c-convex function $g:\mathbb{R}^d \rightarrow \mathbb{R}\cup \{+\infty\}$. Thus, $\mu$-a.e.,
\begin{align}
\nabla g(x) + \nabla_{x}\left(\frac{1}{2}\|x -T(x) \|^2\right) &= 0 \\
\nabla g(x) + x - T(x) &= 0.
\end{align}. 
Thus,
\begin{equation}
    T(x) = x + \nabla g(x).
\end{equation}
Finally, since $u$ and $c(x,y)$ are convex, $g$ must be 1-weakly convex.
\end{proof}

\subsection{Proof of Lemma ~\ref{lemma:approx_fixed_point}}
\begin{proof}
Because $\nabla_yg_{\theta}$ is assumed to be $L_y$ Lipschitz in $y$ it follows that $\nabla_yg_{\theta}$ is continuous in $y$. Therefore, the function
\begin{equation*}
    \tilde{y} \mapsto (\tilde{y} - z) + \nabla_yg_{\theta}(\tilde{y}) 
\end{equation*}
is continuous at all points in its domain. Then, by definition, $\forall \epsilon' > 0$, $\exists \delta' > 0$ such that $\|(\tilde{y} - z) + \nabla_yg_{\theta}(\tilde{y}) - \left( (y^* - z) - \nabla_{y}g_{\theta}(y^*)\right)\| < \epsilon'$ whenever $\|\tilde{y} - y^*\| < \delta'$. Because $y^*$ is an exact minimizer of $\frac{1}{2}\|y-z\|^2 + g_{\theta}(y)$, it follows that whenever $\|\tilde{y} - y^*\| < \delta'$,
\begin{equation*}
\|(\tilde{y} - z) + \nabla_yg_{\theta}(\tilde{y}) - \left( (y^* - z) - \nabla_{y}g_{\theta}(y^*)\right)\| = \|(\tilde{y} - z) + \nabla_yg_{\theta}(\tilde{y}) - 0\| = \|(\tilde{y} - z) + \nabla_yg_{\theta}(\tilde{y})\| < \epsilon'.
\end{equation*}
Thus, the result follows by choosing $\epsilon = \delta'$ corresponding to $\epsilon' = \epsilon_p$.
\end{proof}

\subsection{Proof of Lemma ~\ref{lemma:2nd_moment}}
\begin{proof}
Because $\frac{\partial g}{\partial \theta}$ is $L_{\theta}$-Lipschitz in $\theta$ and by the result of Lemma ~\ref{lemma:approx_fixed_point}, $\exists \epsilon > 0$ such that $\|\tilde{y}_{\theta}(z) - y_{\theta}^*(z)\| < \epsilon$, it follows that
\begin{equation*}
0 \leq \mathbb{E}_z\left[\left\|\frac{\partial g(\tilde{y}_{\theta}(z))}{\partial \theta} - \frac{\partial g(\tilde{y}_{\theta}(z))}{\partial \theta} \right\|^2\right] \leq \mathbb{E}_z \left[ L_g\left\|\tilde{y}_{\theta}(z) - y_{\theta}^*(z)\right\|^2\right] \leq \epsilon L_g.
\end{equation*}
Expanding $\mathbb{E}_z\left[\left\|\frac{\partial g(\tilde{y}_{\theta}(z))}{\partial \theta} - \frac{\partial g(\tilde{y}_{\theta}(z))}{\partial \theta} \right\|^2\right]$, the above inequality becomes
\begin{equation*}
0 \leq \mathbb{E}_z\left[\left\|\frac{\partial g(\tilde{y}_{\theta}(z))}{\partial \theta} \right\|^2\right] -2\mathbb{E}_z\left[\frac{\partial g(\tilde{y}_{\theta}(z))}{\partial \theta}^{\top}\frac{\partial g(y_{\theta}^*(z))}{\partial \theta}\right] + \mathbb{E}_z\left[\left\|\frac{\partial g(\tilde{y}_{\theta}(z))}{\partial \theta}\right\|^2\right] \leq \epsilon L_g.
\end{equation*}
It follows from the above inequality and the assumption that $\int\|x\|^2d\mu(x)$ and $\int\|z\|^2d\nu(z)$ that $\mathbb{E}_z\left[\left\|\frac{\partial g(\tilde{y}_{\theta}(z))}{\partial \theta} \right\|^2\right]$,  $\mathbb{E}_z\left[\frac{\partial g(\tilde{y}_{\theta}(z))}{\partial \theta}^{\top}\frac{\partial g(y_{\theta}^*(z))}{\partial \theta}\right]$, and $\mathbb{E}_z\left[\left\|\frac{\partial g(\tilde{y}_{\theta}(z))}{\partial \theta}\right\|^2\right]$ are all finite. \newline
A near-identical argument shows $\mathbb{E}_x\left[\left\|\frac{\partial g(x_{\theta})}{\partial \theta}\right\|^2\right]$ is finite by picking any two points $x_{\theta,1}, x_{\theta,2}$ such that $\|x_{\theta,1} - x_{\theta,2}\| < \epsilon$ and using the fact that $\frac{\partial g(x_{\theta})}{\partial \theta}$ is $L_g$-Lipschitz.
\end{proof}

\subsection{Proof of Theorem ~\ref{theorem:norm_squared}}
\begin{proof}
Expanding the definition of $\tilde{d}$,
\begin{align}
\|\tilde{d}_{\theta}\|^2 &= \left\|\mathbb{E}_x\left[\dgx\right] - \mathbb{E}_z\left[ \dgytilde \right] \right\|^2 \\
&= \left\| \mathbb{E}_x\left[\dgx \right] \right\|^2 - 2\left\langle \mathbb{E}_x\left[\dgx\right], \mathbb{E}_z\left[ \dgytilde \right] \right\rangle + \left\| \mathbb{E}_z\left[ \dgytilde \right]\right\|^2 \\
&= \mathbb{E}_x\left[\left\| \dgx\right\|^2 \right] - 2\left\langle \mathbb{E}_x\left[ \dgx \right], \mathbb{E}_z\left[ \dgytilde \right] \right\rangle + \mathbb{E}_z\left[ \left\|\dgytilde\right\|^2 \right], 
\label{eq:2nd_moment}
\end{align}
where ~\eqref{eq:2nd_moment} follows by Jensen's inequality. By the result of Lemma ~\ref{lemma:2nd_moment}, $\exists 0 < M_x < +\infty, 0 < M_z < +\infty$ such that $\mathbb{E}_x\left[\left\| \dgx\right\|^2\right] < M_x$ and $\mathbb{E}_z\left[\left\| \dgytilde \right\|^2\right] < M_z$. Applying Jensen's inequality to $\mathbb{E}_x\left[\left\| \dgx\right\|^2\right], \mathbb{E}_z\left[\left\| \dgytilde \right\|^2\right]$ yields
\begin{equation*}
\left\| \mathbb{E}_x\left[\dgx\right] \right\| < \sqrt{M_x} \text{ and } \left\| \mathbb{E}_x\left[\dgytilde\right] \right\| < \sqrt{M_z}.
\end{equation*}
It then follows by the Cauchy-Schwarz inequality that
\begin{equation}
- \left\| \mathbb{E}_x\left[ \dgx \right]\right\| \left\| \mathbb{E}_x\left[\dgytilde\right] \right\| \leq \left\langle \mathbb{E}_x\left[ \dgx \right], \mathbb{E}_z\left[ \dgytilde \right] \right\rangle \leq \left\| \mathbb{E}_x\left[ \dgx \right]\right\| \left\| \mathbb{E}_x\left[\dgytilde\right] \right\|.
\label{eq:cs}
\end{equation}
Combining ~\eqref{eq:2nd_moment} and ~\eqref{eq:cs},
\begin{equation*}
0 \leq \|\tilde{d}_{\theta}\|^2 \leq M_x + M_z + 2\sqrt{M_xM_z} =\tilde{M} < + \infty.
\end{equation*}
An identical argument shows the existence of $M_L$ such that $0 \leq \|\nabla_{\theta}L\|^2 \leq M_L < +\infty$.
\end{proof}

\subsection{Proof of Theorem ~\ref{theorem:lower_bound_on_inner_product}}
\begin{proof}
Expanding the inner product, 
\begin{equation*}
\langle \nabla_{\theta}L, \tilde{d}_{\theta}\rangle = \left\langle \mathbb{E}_x\left[\dgx\right] - \mathbb{E}_z\left[\dgystar\right], \mathbb{E}_x\left[\dgx\right] - \mathbb{E}_z\left[\dgytilde\right]\right\rangle
\end{equation*}
\begin{equation*}
\begin{split}
\langle \nabla_{\theta}L, \tilde{d}_{\theta}\rangle = \left\langle \mathbb{E}_x\left[\dgx\right], \mathbb{E}_x\left[\dgx\right]\right\rangle - \left\langle \mathbb{E}_x\left[\dgx\right], \mathbb{E}_z\left[\dgytilde\right] + \mathbb{E}_z\left[\dgystar\right]\right\rangle + \\ \left\langle\mathbb{E}_z\left[\dgytilde\right],\mathbb{E}_z\left[\dgystar\right]\right\rangle
\end{split}
\end{equation*}

\begin{equation}
\begin{split}
\langle \nabla_{\theta}L, \tilde{d}_{\theta}\rangle = \left\| \mathbb{E}_x\left[\dgx\right]\right\|^2 - \left\langle \mathbb{E}_x\left[\dgx\right], \mathbb{E}_z\left[\dgytilde\right] + \mathbb{E}_z\left[\dgystar\right]\right\rangle + \\\left\langle\mathbb{E}_z\left[\dgytilde\right],\mathbb{E}_z\left[\dgystar\right]\right\rangle.
\label{eq:expand_inner_product}
\end{split}
\end{equation}
$\forall z$, let $\xi_{\theta}(z) = \ytilde - \ystar$ so $\|\xi_{\theta}(z)\| = \epsilon$. Fix $z$. Because the $2^{nd}$ order partial derivatives of $g$ with respect to $\theta$ are assumed to exist, by Taylor's Theorem, $\exists t \in (0,1)$ such that
\begin{equation}
\dgytilde = \dgystar + \nabla_{\theta}^2g(\ystar + t\xi_{\theta}(z))\xi_{\theta}(z).
\end{equation}
Taking expectation with respect to $z$ on both sides, by linearity of expectation,
\begin{equation}
\mathbb{E}_z\left[\dgytilde\right] = \mathbb{E}_z\left[\dgystar\right] + \mathbb{E}_z\left[\nabla_{\theta}^2g(\ystar + t\xi_{\theta}(z))\xi_{\theta}(z)\right].
\end{equation}
Therefore,
\begin{equation*}
    \begin{split}
    \left\langle \mathbb{E}_x\left[\dgx\right],\mathbb{E}_z\left[\dgytilde\right] + \mathbb{E}_z\left[\dgystar\right] \right\rangle = \\
    \left\langle \mathbb{E}_x\left[ \dgx \right],\mathbb{E}_z\left[\dgystar\right] +\mathbb{E}_z\left[\nabla_{\theta}^2g(\ystar + t\xitheta)\xitheta\right] + \mathbb{E}_z\left[\dgystar\right]\right\rangle
    \end{split} 
\end{equation*}

\begin{equation}
    \begin{split}
    \left\langle \mathbb{E}_x\left[\dgx\right], \mathbb{E}_z\left[\dgytilde\right] + \mathbb{E}_z\left[\dgystar\right]\right\rangle= 2\left\langle \mathbb{E}_x\left[\dgx\right], \mathbb{E}_z\left[\dgystar\right]\right\rangle + \\ \left\langle \mathbb{E}_x\left[\dgx\right],\mathbb{E}_z\left[\nabla_{\theta}^2g(\ystar + t\xitheta)\xitheta\right]  \right\rangle
    \end{split}
\label{eq:expand_middle}
\end{equation}
and
\begin{align}
\left\langle \mathbb{E}_z\left[\dgytilde\right], \mathbb{E}_z\left[\dgystar\right]\right\rangle &= \left\langle \mathbb{E}_z\left[\dgystar\right] + \mathbb{E}_z\left[\nabla_{\theta}^2g(\ystar + t\xitheta)\xitheta\right], \mathbb{E}_z\left[\dgystar\right]\right\rangle \\ 
&= \left\| \mathbb{E}_z\left[\dgystar\right]\right\|^2 + \left\langle \mathbb{E}_z\left[\nabla_{\theta}^2g(\ystar + t\xitheta)\xitheta\right], \mathbb{E}_z\left[\dgystar\right]\right\rangle.
\label{eq:expand_right}
\end{align}
Combining ~\eqref{eq:expand_inner_product}, ~\eqref{eq:expand_middle}, and ~\eqref{eq:expand_right} gives
\begin{equation*}
\begin{split}
\langle \nabla_{\theta}L, \tilde{d}_{\theta} \rangle = \left\| \mathbb{E}_x\left[\dgx\right]\right\|^2 - 2\left\langle \mathbb{E}_x\left[\dgx\right],\mathbb{E}_z\left[\dgystar\right]\right\rangle + \left\| \mathbb{E}_z\left[\dgystar\right]\right\|^2 - \\
\left\langle \mathbb{E}_z\left[\nabla_{\theta}^2g(\ystar + t\xi_{\theta}(z))\xitheta\right], \mathbb{E}_x\left[\dgx\right] + \mathbb{E}_z\left[\dgystar\right]  \right\rangle
\end{split}
\end{equation*}
\begin{equation}
\langle \nabla_{\theta}L, \tilde{d}_{\theta} \rangle = \left\| \nabla_{\theta}L\right\|^2 - \left\langle \mathbb{E}_z\left[\nabla_{\theta}^2g(\ystar + t\xi_{\theta}(z))\xitheta\right], \mathbb{E}_x\left[\dgx\right] + \mathbb{E}_z\left[\dgystar\right]  \right\rangle.
\label{eq:complete_the_square}
\end{equation}
By Cauchy-Schwarz and the triangle inequality
\begin{equation*}
\begin{split}
\left\langle \mathbb{E}_z\left[\nabla_{\theta}^2g(\ystar + t\xi_{\theta}(z))\xitheta\right], \mathbb{E}_x\left[\dgx\right] + \mathbb{E}_z\left[\dgystar\right]  \right\rangle \leq \\\left\|\mathbb{E}_z\left[\nabla_{\theta}^2g(\ystar + t\xi_{\theta}(z))\xitheta\right]\right\| \left\|\mathbb{E}_x\left[\dgx\right] + \mathbb{E}_z\left[\dgystar\right]\right\|
\end{split}
\end{equation*}
\begin{equation*}
\begin{split}
\left\langle \mathbb{E}_z\left[\nabla_{\theta}^2g(\ystar + t\xi_{\theta}(z))\xitheta\right], \mathbb{E}_x\left[\dgx\right] + \mathbb{E}_z\left[\dgystar\right]  \right\rangle \leq \\ \mathbb{E}_z\left[\left\|\nabla_{\theta}^2g(\ystar + t\xi_{\theta}(z))\xitheta\right\|\right] \left(\left\|\mathbb{E}_x\left[\dgx\right]\right\| + \left\|\mathbb{E}_z\left[\dgystar\right]\right\| \right)
\end{split}
\end{equation*}
\begin{equation*}
\begin{split}
\left\langle \mathbb{E}_z\left[\nabla_{\theta}^2g(\ystar + t\xi_{\theta}(z))\xitheta\right], \mathbb{E}_x\left[\dgx\right] + \mathbb{E}_z\left[\dgystar\right]  \right\rangle \leq \\ \mathbb{E}_z\left[\left\|\nabla_{\theta}^2g(\ystar + t\xi_{\theta}(z))\right\|\left\|\xitheta\right\|\right]\left(\sqrt{M_x} + \sqrt{M_z}\right)
\end{split}
\end{equation*}
\begin{equation}
\begin{split}
\left\langle \mathbb{E}_z\left[\nabla_{\theta}^2g(\ystar + t\xi_{\theta}(z))\xitheta\right], \mathbb{E}_x\left[\dgx\right] + \mathbb{E}_z\left[\dgystar\right]  \right\rangle \leq \\\sup_{z}\left(\left\|\nabla_{\theta}^2g(\ystar + t\xi_{\theta}(z))\right\|\right)\sup_{z}\left(\left\|\xitheta\right\|\right)\left(\sqrt{M_x} + \sqrt{M_z}\right).
\end{split}
\label{eq:upper_bound_nabla_squared}
\end{equation}
Because $\frac{\partial g(\cdot)}{\partial \theta}$ is assumed to be $L_{\theta}$-Lipschitz, it follows that $\sup_{z}\left(\left\|\nabla_{\theta}^2g(\ystar + t\xitheta)\right\| \right) \leq L_{\theta}$. Then, because for any $z$, $\|\xi_{\theta}(z)\| = \epsilon$~\eqref{eq:upper_bound_nabla_squared} becomes
\begin{equation}
\left\langle \mathbb{E}_z\left[\nabla_{\theta}^2g(\ystar + t\xi_{\theta}(z))\xitheta\right], \mathbb{E}_x\left[\dgx\right] + \mathbb{E}_z\left[\dgystar\right]  \right\rangle \leq L_{\theta}\epsilon\left(\sqrt{M_x} + \sqrt{M_z}\right).
\end{equation}
Thus, because $\epsilon \leq \frac{V\|\nabla_{\theta}L\|^2}{L_{\theta}\left(\sqrt{M_x} + \sqrt{M_z}\right)}$ for some $0 < V < 1$, 
\begin{align}
\left\langle \mathbb{E}_z\left[\nabla_{\theta}^2g(\ystar + t\xi_{\theta}(z))\xitheta\right], \mathbb{E}_x\left[\dgx\right] + \mathbb{E}_z\left[\dgystar\right]  \right\rangle &\leq L_{\theta}\left(\frac{V\|\nabla_{\theta}L\|^2}{L_{\theta}\left(\sqrt{M_x} + \sqrt{M_z}\right)}\right)\left(\sqrt{M_x} + \sqrt{M_z}\right) \\
&\leq V\|\nabla_{\theta}L\|^2.
\label{eq:upper_bound_epsilon}
\end{align}
Hence, combining ~\eqref{eq:complete_the_square} and ~\eqref{eq:upper_bound_epsilon},
\begin{align}
\left\langle \nabla_{\theta}L, \tilde{d}_{\theta} \right\rangle &\geq \|\nabla_{\theta}L\|^2 - V\|\nabla_{\theta}L\|^2 \\
&\geq (1-V)\|\nabla_{\theta}L\|^2 \\
&\geq U \|\nabla_{\theta}L\|^2.
\end{align}
\end{proof}

\subsection{Proof of Lemma ~\ref{lemma:expectation_descent}}
\begin{proof}
First, it needs to be shown that $\nabla_{\theta}L$ is Lipschitz with respect to $\theta$. Using the fact that $\frac{\partial g}{\partial \theta}$ is $L_{\theta}$-Lipschitz with respect to $\theta$, for any $\theta_{j+1},\theta_j \in \mathbb{R}^p$
\begin{align*}
\|\nabla_{\theta}L(\theta_{j+1}) - \nabla_{\theta}L(\theta_j)\| &= \left\| \mathbb{E}_x\left[\frac{\partial g(x_{\theta_{j+1}})}{\partial \theta} \right] - \mathbb{E}_z\left[\frac{\partial g(y_{\theta_{j+1}}^*(z))}{\partial \theta} \right] - \left(\mathbb{E}_x\left[\frac{\partial g(x_{\theta_{j}})}{\partial \theta} \right] - \mathbb{E}_z\left[\frac{\partial g(y_{\theta_{j}}^*(z))}{\partial \theta} \right]\right)\right\| \\
&= \left\| \mathbb{E}_x\left[\frac{\partial g(x_{\theta_{j+1}})}{\partial \theta} - \frac{\partial g(x_{\theta_{j}})}{\partial \theta}\right] + \mathbb{E}_z\left[\frac{\partial g(y_{\theta_{j}}^*(z))}{\partial \theta} -\frac{\partial g(y_{\theta_{j+1}}^*(z))}{\partial \theta}\right]\right\| \\
&\leq \left\| \mathbb{E}_x\left[\frac{\partial g(x_{\theta_{j+1}})}{\partial \theta} - \frac{\partial g(x_{\theta_{j}})}{\partial \theta}\right]\right\| + \left\|\mathbb{E}_z\left[\frac{\partial g(y_{\theta_{j}}^*(z))}{\partial \theta} -\frac{\partial g(y_{\theta_{j+1}}^*(z))}{\partial \theta}\right]\right\| \\
&\leq \mathbb{E}_x\left[ L_{\theta}\left\| \theta_{j+1} - \theta_j\right\| \right] + \mathbb{E}_z\left[L_{\theta}\left\|\theta_{j+1} - \theta_j \right\| \right] \\
&\leq L_{\theta}\left\| \theta_{j+1} - \theta_j\right\| + L_{\theta}\left\| \theta_{j+1} - \theta_j\right\| \\
&\leq 2L_{\theta} \|\theta_{j+1} - \theta_j\|.
\end{align*}
Therefore, $\nabla_{\theta}L$ is $2L_{\theta}$-Lipschitz with respect to $\theta$. \newline
Because the gradient with respect to $\theta$ of $L$, the second order Taylor series expansion of $L$ centered at $\theta=\theta_j$ satisfies, using $\theta_{j+1} = \theta_j - \alpha_j\tilde{d}_{\epsilon_j}(\theta_j)$
\begin{align*}
L(\theta_{j+1}) &\leq L(\theta_j) + \nabla_{\theta}L(\theta_j)^{\top}(\theta_{j+1} - \theta_{j}) + \frac{1}{2}(2L_{\theta})\|\theta_{j+1} - \theta_j\|^2 \\
L(\theta_{j+1}) &\leq L(\theta_j) - \alpha_j\nabla_{\theta}L(\theta_j)^{\top}\tilde{d}_{\xi_j}(\theta_j) + \alpha_j^2L_{\theta}\|\tilde{d}_{\xi_j}(\theta_j)\|^2.
\end{align*}
Taking expectation of both sides above with respect to $\xi_j$ yields
\begin{equation}
\mathbb{E}_{\xi_j}[L(\theta_{j+1})] \leq \mathbb{E}_{\xi_j}[L(\theta_j)] - \alpha_j\mathbb{E}_{\xi_j}\left[\nabla_{\theta}L(\theta_j)^{\top}\tilde{d}_{\xi_j}(\theta_j)\right] + \alpha_j^2L_{\theta}\mathbb{E}_{\xi_j}\left[\|\tilde{d}_{\xi_j}(\theta_j)\|^2\right].
\label{eq:descent_expectation_xi_j}
\end{equation}
Because $\theta_j$ depends only on $\xi_j, \xi_{j-1},...,\xi_0$, taking $\mathbb{E}_{\xi_j}[\cdot]$ only affects the LHS of \eqref{eq:descent_expectation_xi_j}. Thus, ~\eqref{eq:descent_expectation_xi_j} becomes
\begin{equation}
\mathbb{E}_{\xi_j}[L(\theta_{j+1})] \leq L(\theta_j) - \alpha_j\nabla_{\theta}L(\theta_j)^{\top}\tilde{d}_{\xi_j}(\theta_j) + \alpha_j^2L_{\theta}\|\tilde{d}_{\xi_j}(\theta_j)\|^2.
\end{equation}
Applying the result of Theorems ~\ref{theorem:lower_bound_on_inner_product} and ~\ref{theorem:norm_squared} to the above inequality,
\begin{equation}
\mathbb{E}_{\xi_j}[L(\theta_{j+1})] - L(\theta_j) \leq -\alpha_jU\|\nabla_{\theta}L(\theta_j)\|^2 + \alpha_j^2L_{\theta}\tilde{M}.
\label{eq:descent2}
\end{equation}
It can be derived that the RHS above is $\leq 0$ when $\alpha_j$ satisfies $0 < \alpha_j \leq \frac{UM_L}{L_{\theta}\tilde{M}}$:
\begin{align*}
-\alpha_jU\|\nabla_{\theta}L(\theta_j)\|^2 + \alpha_{j}^2L_{\theta}\tilde{M} &\leq 0 \\
\alpha_{j}^2L_{\theta}\tilde{M} &\leq \alpha_jU\|\nabla_{\theta}L(\theta_j)\|^2 \\
\alpha_j &\leq \frac{U\|\nabla_{\theta}L(\theta_j)\|^2}{L_{\theta}\tilde{M}} \\
\alpha &\leq \frac{UM_{L}}{L_{\theta}\tilde{M}}.
\end{align*}
\end{proof}

\textit{Proof of Theorem ~\ref{theorem:expectation_cesaro_convergence}}:
\theoremExpectationCesaroConvergence{app}
\begin{proof}
Taking the total expectation of ~\eqref{eq:descent2}, we have
\begin{equation}
\mathbb{E}[L(\theta_{j+1})] - \mathbb{E}[L(\theta_j)] \leq -\alpha_jU\mathbb{E}[\|\nabla_{\theta}L(\theta_j)\|^2] + \alpha_j^2L_{\theta}\tilde{M}.
\label{eq:conv1}
\end{equation}
Setting $j=0$ in ~\eqref{eq:conv1}, we have
\begin{align}
&\mathbb{E}[L(\theta_{1})] - \mathbb{E}[L(\theta_0)] \leq 
 -\alpha_jU\mathbb{E}[\|\nabla_{\theta}L(\theta_0)\|^2] + \alpha_j^2L_{\theta}\tilde{M} \\
&\mathbb{E}[L(\theta_{1})] \leq \mathbb{E}[L(\theta_0)]
-\alpha_jU\mathbb{E}[\|\nabla_{\theta}L(\theta_0)\|^2] + \alpha_j^2L_{\theta}\tilde{M}.
\label{eq:conv2}
\end{align}
Setting $j=1$ in ~\eqref{eq:conv1} and applying ~\eqref{eq:conv2} 
\begin{align*}
&\mathbb{E}[L(\theta_{2})] - \mathbb{E}[L(\theta_1)] \leq 
 -\alpha_jU\mathbb{E}[\|\nabla_{\theta}L(\theta_1)\|^2] + \alpha_j^2L_{\theta}\tilde{M}\\
&\mathbb{E}[L(\theta_{2})] \leq \mathbb{E}[L(\theta_1)]
-\alpha_jU\mathbb{E}[\|\nabla_{\theta}L(\theta_1)\|^2] + \alpha_j^2L_{\theta}\tilde{M}\\
&\mathbb{E}[L(\theta_{2})] \leq \mathbb{E}[L(\theta_0)]
-U\sum_{j=0}^{1}\alpha_j\mathbb{E}[\|\nabla_{\theta}L(\theta_j)\|^2] + L_{\theta}\tilde{M}\sum_{j=0}^{1}\alpha_{j}^2\\
&\mathbb{E}[L(\theta_{2})] - \mathbb{E}[L(\theta_0)]\leq 
-U\sum_{j=0}^{1}\alpha_j\mathbb{E}[\|\nabla_{\theta}L(\theta_j)\|^2] + L_{\theta}\tilde{M}\sum_{j=0}^{1}\alpha_{j}^2 \\
&\vdots  \\
&\mathbb{E}[L(\theta_{K})] - \mathbb{E}[L(\theta_0)]\leq 
-U\sum_{j=0}^{K-1}\alpha_j\mathbb{E}[\|\nabla_{\theta}J_x(\theta_j)\|^2] + L_{\theta}\tilde{M}\sum_{j=0}^{K-1}\alpha_{j}^2.
\end{align*}
Since $L(\theta)$ is bounded from below by $L_{inf}$, algebraically rearranging the last line above yields
\begin{align}
&\sum_{j=0}^{K-1}\alpha_j\mathbb{E}[\|\nabla_{\theta}L(\theta_j)\|^2] \leq 
\frac{\mathbb{E}[L(\theta_0)] - L_{inf}}{U} + \frac{L_{\theta}\tilde{M}}{U}\sum_{j=0}^{K-1}\alpha_{j}^2\\
&\frac{1}{A_K}\sum_{j=0}^{K-1}\alpha_j\mathbb{E}[\|\nabla_{\theta}L(\theta_j)\|^2] \leq \frac{\mathbb{E}[L(\theta_0)] - L_{inf}}{UA_K}+ \frac{L_{\theta}\tilde{M}}{UA_K}\sum_{j=0}^{K-1}\alpha_{j}^2.
\label{eq:conv3}
\end{align}
Using linearity of expectation, ~\eqref{eq:conv3} becomes
\begin{equation}
\mathbb{E}\left[\frac{1}{A_K}\sum_{j=0}^{K-1}\alpha_j\|\nabla_{\theta}L(\theta_j)\|^2\right] \leq \frac{\mathbb{E}[L(\theta_0)] - L_{inf}}{UA_K}+ \frac{L_{\theta}\tilde{M}}{UA_K}\sum_{j=0}^{K-1}\alpha_{j}^2.
\label{eq:conv4}
\end{equation}
Hence, since $\lim_{K \rightarrow \infty} A_K = \infty$, and $\sum_{j=0}^{\infty}\alpha_{j}^2$ converges, we have
\begin{align*}
&\lim_{K \rightarrow \infty}\mathbb{E}\left[\frac{1}{A_K}\sum_{j=0}^{K}\alpha_j\|\nabla_{\theta}J_x(\theta_j)\|^2\right] \\
&\leq \lim_{K \rightarrow \infty}\left[\frac{\mathbb{E}[L(\theta_0)]-L_{inf} + L_{\theta}\tilde{M}{\sum_{j=0}^{K-1}\alpha_{j}^2}}{U A_K}\right] \\
&= 0.
\end{align*}
\end{proof}

\subsection{Proof of Theorem ~\ref{theorem:expectation_liminf}}
\begin{proof}
Suppose, for contradiction that, for some $a > 0$
\begin{equation}
\liminf_{j \rightarrow \infty} \mathbb{E}\left[\left\| \nabla_{\theta}L(\theta_j)\right\|^2 \right] = a.
\end{equation}
Let $K \in \mathbb{N}$ and $A_K = \sum_{k=0}^{K-1}\alpha_k$. Then, using linearity of expectation
\begin{equation*}
\mathbb{E}\left[ \frac{1}{A_K}\sum_{j=0}^{K-1}\alpha_k \left\| \nabla_{\theta}L(\theta_j) \right\|^2 \right] = \frac{1}{A_K}\sum_{j=0}^{K-1}\alpha_j\mathbb{E}\left[ \left\| \nabla_{\theta}L(\theta_j) \right\|^2\right].
\end{equation*}
Taking the liminf as $K \rightarrow \infty$ on both sides of the above equation yields a contradiction, as the LHS converges to 0 but the RHS diverges. Hence, by contradiction the result follows. 
\end{proof}

\subsection{Proof of Corollary ~\ref{corollary:convergence_probability}}
\begin{proof}
This proof is similar to that of Theorem 4.11 in ~\cite{bottou2018optimization}, but we will include it here to be complete. Let $\epsilon > 0$ and let $\mathbb{E}[\cdot]$ represent total expectation. By Markov's inequality and the law of total expectation, also known as the tower property,
\begin{align}
P(\|\nabla_{\theta}L(\theta_{j(K)})\| &\geq \epsilon) = P(\|\nabla_{\theta}L(\theta_{j(K)})\|^2 \geq \epsilon^2) \\
&\leq \frac{1}{\epsilon^2}\mathbb{E}[\mathbb{E}_{j(K)}[\nabla_{\theta}L(\theta_{j(K)})]].
\label{eq:convergence_in_prob}
\end{align}
By the proof of Theorem ~\ref{theorem:expectation_cesaro_convergence} we have $\lim_{K \rightarrow \infty}\mathbb{E} \left[\sum_{j=0}^{K-1}\alpha_j \|\nabla_{\theta}L(\theta_j)\|^2 \right] < \infty$. Therefore, we must have $\lim_{j \rightarrow \infty}\mathbb{E}\left[\alpha_j\|\nabla_{\theta}L(\theta_j)\|^2\right] = 0$. Thus, by ~\eqref{eq:convergence_in_prob},
\begin{equation}
\lim_{K \rightarrow \infty} P(\|\nabla_{\theta}L(\theta_{j(K)})\| \geq \epsilon) \leq \lim_{K \rightarrow \infty } \frac{1}{\epsilon^2}\mathbb{E}[\mathbb{E}_{j(K)}[\nabla_{\theta}L(\theta_{j(K)})]] = 0.
\end{equation}
Since the choice of $\epsilon > 0$ was arbitrary, this holds $\forall \epsilon > 0$, proving convergence in probability.
\end{proof}





\section{Implementation Details}\label{appen:implementation_details}

\subsection{Datasets}

\subsubsection{Synthetic Gaussians}
For the synthetic high-dimensional Gaussian experiments in Section~\ref{sec:high_dim_gaussians_rewrite}, we use the data generation procedure of \cite{park2025neural}, which follows the standard protocol of \cite{korotin2021neural} for constructing Gaussian pairs with controlled covariance spectra.

In this setup, each covariance matrix is defined as
\[
\Sigma = Q \Lambda Q^\top,
\]
where $Q$ is an orthogonal matrix obtained via QR decomposition of a Gaussian random matrix, and $\Lambda$ contains eigenvalues sampled from a log-uniform distribution,
\[
\lambda_i \sim \exp(\mathcal{U}(a,b)).
\]
This yields well-conditioned but heterogeneous Gaussian covariances.

Given two Gaussian distributions $\mathcal{N}(0, \Sigma_0)$ and $\mathcal{N}(0, \Sigma_1)$, the optimal transport map is linear, $T(x)=\Gamma x$, where
\[
\Gamma = \Sigma_0^{-1/2}
\left(\Sigma_0^{1/2} \Sigma_1 \Sigma_0^{1/2}\right)^{1/2}
\Sigma_0^{-1/2}.
\]

Samples from $\mathcal{N}(0, \Sigma_0)$ and $\mathcal{N}(0, \Sigma_1)$ are drawn following the same procedure as in \cite{park2025neural}, with $n=10^5$ samples for both training and evaluation. Independent test samples are used to evaluate transport accuracy, where the ground-truth map $T(x)=\Gamma x$ is applied to source samples, and $\Gamma^{-1}$ is used for inverse-consistency checks.

\subsubsection{UCI Physics Data}
The UCI physics dataset is used following the preprocessing protocol of \cite{onken2021ot}, and we adopt the same train/validation/test split and normalization procedure as in the original work.

\paragraph{GAS Dataset}
For the GAS dataset, 
the raw data is loaded from \texttt{ethylene\_CO.csv}, and the variables \texttt{Meth}, \texttt{Eth}, and \texttt{Time} are removed. Highly correlated features are iteratively removed by discarding variables whose pairwise correlation exceeds $0.98$. After this cleaning step, each feature is standardized using z-score normalization based on the training data statistics. The dataset is then split into training, validation, and test sets using a $80/10/10$ split.

\paragraph{POWER Dataset}
For the POWER dataset, we use the preprocessed data from \texttt{data.npy}. Two features are removed as in the original preprocessing pipeline, and additional synthetic noise is added to several variables following the procedure of prior work. The dataset is then randomly shuffled and split into training, validation, and test sets using a $80/10/10$ ratio.

\paragraph{HEPMASS Dataset}
For the HEPMASS dataset, we follow the preprocessing pipeline of prior work on tabular density estimation. We use the split provided in \texttt{1000\_train.csv} and \texttt{1000\_test.csv}, and restrict the data to the positive class only by removing background noise samples (class label $0$). In addition, the label column and one redundant feature column in the test set are removed.

We consider three normalization schemes: static normalization, min-max scaling, and SVD-based whitening. In the experiments reported in the main paper, we use the \textit{min-max scaling} variant (denoted as \texttt{scale}), where each feature is normalized as
\[
x' = \frac{x - \mu}{s}, \quad
\mu = \frac{\max(x) + \min(x)}{2}, \quad
s = \frac{\max(x) - \min(x)}{2}.
\]
Statistics are computed using the training set only. The resulting transformation maps each feature approximately to $[-1, 1]$. We further remove features with highly repetitive values following the procedure in the original preprocessing code.

\paragraph{MINIBOONE Dataset}
For the MINIBOONE dataset, we use the preprocessed version provided as a NumPy array. The dataset is split into training, validation, and test sets using an $80/10/10$ ratio.

We apply min-max scaling (denoted as \texttt{scale}) using statistics computed on the combined training and validation set:
\[
x' = \frac{x - \mu}{s}, \quad
\mu = \frac{\max(x) + \min(x)}{2}, \quad
s = \frac{\max(x) - \min(x)}{2}.
\]
This normalization is applied independently to each feature, resulting in an approximate $[-1, 1]$ range per dimension.

\paragraph{BSDS300 Dataset}
The BSDS300 dataset consists of natural image patches extracted from the Berkeley Segmentation Dataset. We use the standard train/validation/test split provided in the preprocessed HDF5 file \texttt{BSDS300.hdf5}, following prior work on image-based optimal transport and density modeling.

Each sample corresponds to a vectorized image patch, and no additional feature engineering is applied. We directly load the precomputed splits for training, validation, and testing.

Since the data consists of image patches, we additionally record the corresponding spatial resolution, given by
\[
\text{image size} = \left[\sqrt{d+1}, \sqrt{d+1}\right],
\]
where $d$ denotes the dimensionality of the vectorized patch. No further normalization beyond the provided preprocessing is applied.

\subsubsection{Synthetic Class-Conditional OT Datasets}
We construct three synthetic benchmarks for class-conditional optimal transport using mixtures of two-dimensional Gaussian clusters arranged on a circular manifold. All datasets are generated on a ring of radius $r=0.9$ with isotropic Gaussian noise $\epsilon = 0.07$. Each class corresponds to a subset of Gaussian components placed at fixed angular positions.

\paragraph{Crossed Ring Gaussians}
In the Crossed Ring Gaussian dataset, each distribution consists of two Gaussian modes placed at distinct angular positions on the unit circle. Specifically, the source distribution is defined by two clusters centered at angles $0$ and $\pi/2$, while the target distribution is defined by clusters at angles $\pi/4$ and $3\pi/4$. This creates a crossed transport structure between classes. The total number of samples is $N = 50{,}000$, with equal allocation per mode.

\paragraph{Horizontal Swapped Gaussians}
We construct an additional synthetic benchmark consisting of two one-dimensional Gaussian clusters embedded in $\mathbb{R}^2$. The dataset is designed to evaluate class-conditional optimal transport under near-degenerate geometric structure.

The source distribution consists of two Gaussian components centered at $(-0.23, 0)$ and $(0.23, 0)$, while the target distribution consists of two components centered at $(0.7, 0)$ and $(-0.7, 0)$, respectively. This induces a horizontal swapping structure along the $x$-axis.

Formally, each component is generated as
\[
x = \mu + \epsilon \cdot \mathcal{N}(0, I_2),
\]
where $\mu \in \mathbb{R}^2$ denotes the cluster center and $\epsilon = 0.05$ controls the noise level. The total number of samples is $N = 50{,}000$, equally split across components.

This dataset tests the ability of class-conditional OT methods to recover correct pairwise matching in a low-dimensional, highly structured setting.
\paragraph{Four-Mode Ring Gaussians}
The Four-Mode Ring Gaussian dataset consists of four Gaussian components per distribution placed uniformly on the unit circle. The source distribution uses angles $\{0, \pi/2, \pi, 3\pi/2\}$, while the target distribution uses $\{\pi/4, 3\pi/4, 5\pi/4, 7\pi/4\}$. This results in a dense multimodal matching problem with $N = 5{,}000$ total samples.

For all datasets, each mode is sampled as
\[
x = r \cdot (\cos\theta, \sin\theta) + \epsilon \cdot \mathcal{N}(0, I_2),
\]
where $r=0.9$ and $\epsilon=0.07$. Samples are concatenated across modes to form the full empirical distributions. The resulting datasets are used to evaluate class-conditional optimal transport under structured multimodal alignments.

\subsubsection{Image Data}
Normalization is performed using a min-max scaling transformation computed from the combined training and validation set:
\[
x' = \frac{x - \mu}{s}, \quad
\mu = \frac{\max(x) + \min(x)}{2}, \quad
s = \frac{\max(x) - \min(x)}{2}.
\]
This maps each feature approximately into the range $[-1, 1]$.
All images are resized to $32 \times 32$ pixels and normalized to $[0, 1]$.
For training the OT map, we construct a \textbf{class-paired dataset} by randomly
drawing matched $(x_{\mathrm{src}}, x_{\mathrm{tgt}}, k)$ triples where both images
belong to the same class $k \in \{0, \ldots, 9\}$, drawn uniformly at random from
FashionMNIST (source) and MNIST (target) training sets (60{,}000 images each).

\subsection{Kantorovich Potential Network}
For the high-dimensional Gaussian experiments in Section~\ref{sec:high_dim_gaussians_rewrite}, we employ the DenseICNN architecture, a fully connected neural network augmented with input-quadratic skip connections, in order to ensure a fair comparison with the baseline models of \cite{korotin2021neural}. When implementing our method and the NCF baseline, we remove the nonlinearity constraints that are typically imposed to enforce convexity.
Following \cite{korotin2021neural}, we adopt the network architecture DenseICNN[1; max(2$d$,64), max(2$d$,64), max($d$,32)] for a $d$-dimensional problem. The model is optimized using the Adam optimizer with a fixed learning rate of $10^{-4}$, independent of the input dimension.

The same network architecture is also used for the 2D class-conditional optimal transport experiments in Section~\ref{sec:uci_physics}, Section~\ref{sec:ccot_2d}, and ablation studies in Section \ref{sec:ablation}.

For the image translation experiments in Section \ref{sec:mnist_fmnist}, the potential $g_\theta(x, k)$ is parameterized by a scalar-valued network
taking as input the concatenation $[x;\, \mathbf{c}_k] \in \mathbb{R}^{25}$,
where $\mathbf{c}_k \in \{0,1\}^{10}$ is the class one-hot vector.
The network decomposes as
\begin{equation}
  g_\theta(x, k)
  = \underbrace{\tfrac{1}{2}\, x^\top (A^\top\! A)\, x + b^\top x + c}_{%
      \text{quadratic}}
  + \underbrace{w^\top \Phi(x, k)}_{\text{nonlinear}},
\end{equation}
where $A \in \mathbb{R}^{25 \times 25}$ is Xavier-initialized, enforcing a
positive semi-definite quadratic component.
The nonlinear branch $\Phi$ is a residual network: a linear opening layer
$\mathbb{R}^{25} \to \mathbb{R}^{512}$, followed by four hidden layers of
width $512$ with residual connections weighted by $h = \tfrac{1}{4}$, and a
linear readout. The activation function is softplus throughout.

\subsection{Implicit Fixed-Point Solver}

The proximal fixed point
$y^\star(z, k) = \operatorname{prox}_{g_\theta}(z, k)$ is computed by iterating
\begin{equation}
  y^{(t+1)}
  = y^{(t)} - \alpha\!\left(y^{(t)} - z + \nabla_y g_\theta(y^{(t)}, k)\right),
\end{equation}
with step size $\alpha = \min\!\bigl(0.1/\hat{L},\; 10^{-2}\bigr)$, where
$\hat{L}$ is the spectral norm of the Hessian of $g_\theta$ estimated via
20 steps of power iteration.
Iterations run until
$\|\nabla_y g_\theta(y^\star, k) + y^\star - z\|_\infty < 10^{-3}$
or for at most $10^4$ steps.
The previous batch's fixed point is used as a warm start.

\subsection{OT Map Training}

The OT map is trained to minimize
\begin{equation}
  \mathcal{L}(\theta)
  = \mathbb{E}_{\mu_k}\!\left[g_\theta(x,k)\right]
  - \mathbb{E}_{\nu_k}\!\left[g_\theta\!\left(y^\star(z,k),\, k\right)\right]
  + \lambda_{\mathrm{MMD}}\,\widehat{\mathrm{MMD}}^2\!\left(T(x),\, z\right),
\end{equation}
where $T(x) = x + \nabla_x g_\theta(x, k)$ is the forward transport map
and $\widehat{\mathrm{MMD}}^2$ is the unbiased multi-scale RBF kernel
estimator. We use five kernel bandwidths,
$\sigma^2 \in \{0.25,\, 0.5,\, 1.0,\, 2.0,\, 4.0\} \times \sigma^2_{\mathrm{median}}$, which are only applied in the image-based experiments in Section~\ref{sec:mnist_fmnist}. Consequently, we set $\lambda_{\mathrm{MMD}} = 0$ for the experiments in Sections~\ref{sec:high_dim_gaussians_rewrite}, \ref{sec:uci_physics}, \ref{sec:ccot_2d}, and \ref{sec:ablation}.

\begin{table}[h]
\centering
\begin{tabular}{lc}
\toprule
\textbf{Hyperparameter} & \textbf{Value} \\
\midrule
Epochs                        & 20 \\
Batch size                    & 64 \\
MMD subsample size            & 5000 \\
$\lambda_{\mathrm{MMD}}$      & $10^{-3}$ \\
Fixed-point tolerance $\tau$  & $10^{-3}$ \\
Max fixed-point iterations    & $10^4$ \\
Optimizer                     & Adam \\
Learning rate                 & $10^{-4}$ \\
LR schedule                   & ReduceLROnPlateau (factor $0.5$, patience $5$) \\
Minimum LR                    & $10^{-7}$ \\
Floating-point precision      & float64 \\
\bottomrule
\end{tabular}
\caption{OT map training hyperparameters.}
\end{table}

\subsection{Conditional VAE for Image Task}

\paragraph{Architecture}

Both the FashionMNIST and MNIST VAEs share identical architecture.
The \textbf{encoder} consists of four strided convolutional layers
(kernel $4\times4$, stride $2$, padding $1$), with channel widths
$1 \to 64 \to 128 \to 256 \to 512$, each followed by BatchNorm and
LeakyReLU($0.2$).
The flattened $512\times2\times2$ feature map is concatenated with the
class one-hot vector and projected to the mean $\mu$ and log-variance
$\log\sigma^2$ of the latent distribution via linear layers, yielding a
latent dimension of $d = 15$.

The {decoder} takes the concatenation of the latent code and class
one-hot as input, projects it to a $512\times2\times2$ feature map, and
upsamples through four transposed convolutional layers (kernel $4\times4$,
stride $2$) with channel widths $512 \to 256 \to 128 \to 64 \to 1$.
Each upsampling block is followed by GroupNorm(16) and three residual blocks
(two $3\times3$ convolutions with GroupNorm and a skip connection).
The final output is passed through a Sigmoid activation.
GroupNorm is used in the decoder (instead of BatchNorm) to avoid
train/eval distribution mismatch.


\paragraph{VAE Training Hyperparameters}

\begin{table}[h]
\centering
\begin{tabular}{lcc}
\toprule
\textbf{Hyperparameter} & \textbf{FashionMNIST VAE} & \textbf{MNIST VAE} \\
\midrule
Latent dimension $d$      & 15   & 15  \\
$\beta_{\max}$            & 1.0  & 0.1 \\
KL warmup epochs          & 20 (linear ramp) & 0 \\
Epochs                    & 1000 & 1000 \\
Batch size                & 128  & 128  \\
Optimizer                 & Adam & Adam \\
Learning rate             & $10^{-3}$ & $10^{-3}$ \\
LR schedule               & \multicolumn{2}{c}{ReduceLROnPlateau (factor $0.5$, patience $100$)} \\
Minimum LR                & \multicolumn{2}{c}{$10^{-7}$} \\
\bottomrule
\end{tabular}
\caption{VAE training hyperparameters.}
\end{table}
Both VAEs are frozen after pre-training; their parameters are not updated
during OT map training.


\bibliographystyle{siamplain}
\bibliography{mybib}

\end{document}